\documentclass[12pt]{amsart}
\usepackage[english]{babel}
\usepackage{amsmath}
\usepackage{amsthm}
\usepackage{amssymb}
\usepackage{bbm}
\usepackage{mathrsfs}
\usepackage{enumerate}
\usepackage[notcite, final, notref]{showkeys}
\usepackage{dsfont}

\usepackage[dvips]{color}
\newtheorem{theorem}{Theorem}[section]

\newtheorem{lemma}[theorem]{Lemma}
\newtheorem{proposition}[theorem]{Proposition}
\newtheorem{corol}[theorem]{Corollary}
\newtheorem{definition}[theorem]{Definition}

\theoremstyle{definition}
\newtheorem{remark}[theorem]{Remark}
\newtheorem{example}[theorem]{Example}

\newcommand{\w}{\omega}

\usepackage{xcolor}

\newcommand{\N}{\mathbb{N}}
\newcommand{\R}{\mathbb{R}}
\newcommand{\C}{\mathbb{C}}

\newcommand{\E}{\mathbb{E}}

\newcommand{\cS}{\mathcal{S}}

\newcommand{\cF}{\mathcal{F}}

\newcommand{\cH}{\mathcal{H}}

\newcommand{\floor}[1]{\lfloor #1 \rfloor}

\newcommand{\ov}{\overline}
\newcommand{\un}{\underline}

\newcommand{\Cl}{{C \kern -0.1em \ell}}

\newcommand{\inner}[1]{\left\langle  #1 \right\rangle }
\newcommand{\Inner}[1]{\Big\langle  #1 \Big\rangle }

\usepackage{hyperref}

\title[Generalized Fock space and moments]{Generalized Fock space and moments}

\author[D. Alpay]{Daniel Alpay}
\address{(DA) Department of Mathematics\\
Chapman University\\
One University Drive
Orange, California 92866\\
USA}
\email{alpay@chapman.edu}

\author[P. Cerejeiras]{Paula Cerejeiras}

\author[U. K\"ahler]{Uwe K\"ahler}
\address{(PC) and (UK):  CIDMA - Center for Research and Development in Mathematics and
Applications, Department of Mathematics,\newline University of Aveiro \newline Campus Universit\'ario de Santiago \newline 3810-193 Aveiro, Portugal}
\email{pceres@ua.pt}
\email{ukaehler@ua.pt}

\date{}

\begin{document}

\begin{abstract}
In this paper we develop a framework to extend the theory of generalized stochastic processes
in the Hida white noise space to more general probability spaces which include the
grey noise space. To obtain a Wiener-It\^o expansion we recast it as a moment problem and calculate the moments explicitly. We further show the importance of a
family of topological algebras called strong algebras in this context. Furthermore we show the applicability of our approach to the study of stochastic processes.
\end{abstract}

\maketitle


\tableofcontents

\section{Introduction}
\setcounter{equation}{0}
Infinite dimensional analysis and its applications to the theory of generalized stochastic processes originate with the work of T. Hida; see
\cite{MR1244577,MR2444857,MR1408433,MR1387829}. In the one variable setting, the key in Hida's work is the use of the Bochner-Minlos theorem applied to the positive definite Gaussian function $e^{-\frac{\|s\|^2}{2}}$ where $s$ varies in the space of real-valued Schwartz functions $\cS(\R)$ and $\| \cdot \|$ is the $\mathbf L^2(\mathbb R,dx)$ norm, writing
\begin{equation}
\label{montmartre}
e^{-\frac{\|s\|^2}{2}}=\int_{{\mathcal S}^\prime(\R)}e^{i\langle \w,s\rangle}dP(\w).
\end{equation}
In this expression, $\mathcal S^\prime(\R)$ is the space of real-valued tempered distributions, the brackets denote the duality between $\mathcal S(\R)$ and $\mathcal S^\prime(\R)$, and $P(\w)$ is a
probability measure (called white noise measure) on the cylindrical sigma-algebra $\mathcal C$ of $\mathcal S^\prime(\R)$. The probability space $(\mathcal S'(\R),\mathcal C, dP)$ is then called Hida
white noise space. Furthermore, the Lebesgue space $\mathbf L^2(\mathcal S^\prime(\R),\mathcal C, dP)$ is shown to be isomorphic in a natural sense to the Fock space associated to
$\ell^2(\mathbb N,\mathbb C)$, i.e. to the reproducing kernel Hilbert space with reproducing kernel
\begin{equation}
\label{nfock}
e^{\sum_{n=1}^\infty z_n \overline w_n} = \prod_{n=1}^\infty e^{z_n \overline w_n}.
\end{equation}
We further note that the space $\mathbf L^2(\mathcal S^\prime(\R),\mathcal C, dP)$ can be imbedded in a number of Gelfand triples, consisting of stochastic test
functions and stochastic distributions. In one of these triples introduced by Kondratiev, the corresponding space of stochastic distributions has a topological
algebra structure and possesses a number of inequalities proved by V\r{a}ge
\cite{vage96}, \cite[Corollary 1, p. 57]{MR2714906}.\smallskip

In the above Gaussian case, there is a lucky coincidence: the function $e^z$ generates both the white noise measure (via \eqref{montmartre}) and the Fock space (via \eqref{nfock}), and the reproducing kernel of the Fock space is a product of positive definite functions. Furthermore, due to the multiplicative property of the exponential the resulting probability measure is a product measure which leads in a natural way to tensor product representations, in particular the basis in terms of orthogonal polynomials is immediately given by products of standard Hermite polynomials in the corresponding variables. Unfortunately, for more general probability measures this is not anymore true. This leads to the principal problem of obtaining Wiener-It\^o chaos expansions.
Many attempts were made to generalize the case of white noise. The main idea was to move from the Gaussian measure to Poisson measure, where the Hermite polynomials appearing in the white noise case were replaced by Appell systems as biorthogonal systems.
There were also several attempts to preserve more of the structure of the white noise case by looking for generalizations to more Gaussian-like measures
(like for instance obtained when using the Mittag-Leffler function rather than the exponential), which includes the case of grey noise
(randomized Gaussian noise) and generalized grey noise with applications to fractional Brownian motion~\cite{MR3464056, drumond-oliveira-silva08, Bock2019a}. The idea behind these attempts is that the Mittag-Leffler
function still shares many properties of the exponential. Nevertheless even in this case preliminary results in this direction fall short of creating the necessary
framework to allow a treatment similar to the white noise case, moving instead in the direction of using Appell systems~\cite{MR3315581}. Among other reasons this is
due to the underestimation of the crucial role of a family of topological algebras, called strong algebras, in the classic white noise case. \smallskip

In this paper, still using the Bochner-Minlos theorem, we construct a family of probability spaces, which include the white noise space (Gaussian measure) and the grey noise space or generalized grey noise space (Mittag-Leffler measure), and for which the associated (counterpart of) the Fock space has still the above multiplicative structure. More precisely, both the white noise space and its generalization, the grey noise space, are special instances of
probability spaces obtained by applying the Bochner-Minlos theorem to a specific family of positive definite functions, which we call $ML$ and now introduce:

\begin{definition}[The class $ML$] We denote by $ML$ the class of all entire functions $\varphi(z)=\sum_{n=0}^\infty \varphi_n z^n$ satisfying the following conditions:
\begin{enumerate}[i)]
\item $\varphi(0)=1$ and $\varphi^\prime(0)>0$.
\item $\varphi(z\overline{w})$ is a positive definite function on $\mathbb C;$
\item $\varphi(-\|\cdot\|^2/2)$ is a positive definite function on $\mathcal S({\mathbb R})$.
\end{enumerate}
\end{definition}

We note that the function $s\mapsto\varphi(-\|\cdot\|^2/2)$ is continuous with respect to the $\mathbf L^2(\mathbb R,dx)$ topology, and
hence with respect to the Fr\'echet topology on $\mathcal S({\mathbb R})$.\smallskip

As we will see below in Proposition \ref{totoche1}, we have $\varphi_n\ge 0$ for $n\in\mathbb N$.\smallskip

Basic examples of functions in this class are the exponential function $\varphi(z) = e^z,$ linked to Gaussian processes and the white noise space, and the Mittag-Leffler function (entire of order $1/ \alpha$)
$$E_{\alpha}(z) = \sum_{k=0}^\infty \frac{z^k}{\Gamma(\alpha k + 1)}, \quad \alpha >0,$$ on which the grey noise space analysis is based (see the thesis \cite{J2015} or paper~\cite{MR3315581}).\\

We develop the counterpart of white noise space theory for the class $ML$, having as a special case the grey noise space setting.
An important point is that the counterpart of the Fock space is now the reproducing kernel Hilbert space with reproducing kernel

\begin{equation}
\label{fockphi}
K_\varphi(z,w)=\prod_{j=1}^\infty \varphi(z_j\overline{w_j}),
\end{equation}
defined for the sequences $(z_1,z_2,\ldots)$ such that the product
\[
\prod_{j=1}^\infty \varphi(|z_j|^2)
\]
converges. See Proposition \ref{corona123}.\\

The paper consists of eight sections besides this introduction and we now describe its content. Section \ref{sec2} focuses on the family $ML$ and on the classes of
probability spaces obtained from its functions. Moments of the corresponding probability measures and orthogonal polynomials are considered in Section \ref{sec3}.
In Section 4 we study the reproducing kernel Hilbert space associated to an infinite product of positive definite functions, while in Section 5 we use these results
to consider the Fock space in the present setting. In particular we study the role of the Gelfond-Leontiev operator $\partial^{\varphi}$ of fractional differentiation. In view of applications to generalized stochastic processes we review in Section 6 the main aspects of the above mentioned family of topological algebras, called strong algebras. The last three sections are devoted to applications to the theory of stochastic processes and their derivatives.

\section{Grey noise space and the family $ML$}
\setcounter{equation}{0}
\label{sec2}
\subsection{Framework}


\vspace{15pt}

We start with the following proposition.

\begin{proposition}
Let $\varphi\in ML$, with power series expansion $\varphi(z)=1+\sum_{n=1}^\infty \varphi_nz^n$. Then, $\varphi_n\ge 0$ for $n=1,2,\ldots$
\end{proposition}
\begin{proof}
The condition $\varphi(z\overline{z})\ge 0$ implies that the $\varphi_n$ are real. Write then
\[
\varphi(z\overline{w})=\varphi_+(z\overline{w})-(-\varphi_-(z\overline{w}))
\]
where the first sum contains all the terms with positive coefficients $\varphi_n$ and the second sum contains all the terms with strictly negative coefficients $\varphi_n$.
This expresses $\varphi(z\overline{w})$ as a difference of two positive definite functions, with associated reproducing kernel Hilbert space having a trivial intersection. Hence, the second sum is empty
and $\varphi_-(z)=0$.
\end{proof}

\begin{proposition}
\label{totoche1}
Let $\varphi, \psi \in ML$, and let $p_1, p_2$ denote finite probability distributions such that $p_1+p_1=1$. Then:
\begin{enumerate}
\item $p_1 \varphi + p_2 \psi \in ML;$
\item $\varphi \psi \in ML;$
\item $\frac{\varphi ( \psi ( \cdot ))}{\varphi(1)}  \in ML.$
\end{enumerate}
\end{proposition}

\begin{proof}
Continuity with respect to the Fr\'echet topology of the various functions is immediate.
The first claim holds then since a sum of positive definite functions is still positive definite, together with the fact that $p_1 \varphi(0) + p_2 \psi(0)= p_1 + p_2=1$. The second
and third claims follow from the fact that a product of positive definite functions is still positive definite. The claim on the composition uses both facts and the previous proposition, writing
\[
\varphi(\psi(z\overline{w}))=\sum_{k=0}^\infty \varphi_k(\psi(z\overline{w}))^k,\quad{with}\quad \varphi(z)=1+\sum_{n=1}^\infty \varphi_nz^n.
\]
\end{proof}

For instance the functions
\begin{equation}
e^{(e^z-1)},\quad e^{(e^{(e^{z}-1)}-1)},
\end{equation}
the Mittag-Leffler function
\[
E_{\alpha}(z) = \sum_{k=0}^\infty \frac{z^k}{\Gamma(\alpha k + 1)}, \quad \alpha >0,
\]
(which corresponds to the grey noise space) and
\[
E_{\alpha}(e^z-1) = \sum_{k=0}^\infty \frac{(e^z-1)^k}{\Gamma(\alpha k + 1)}, \quad \alpha >0,
\]
all belong to $ML$.\\

We remark that when $\varphi$ is the exponential function the probability measure $P_{\varphi}$ on $\cS'(\R)$ is the Gaussian measure.
The triple $(\cS'(\R), \mathcal C, P_\varphi)$ is then called the $1-$dimensional white noise probability space. White noise analysis is based on the triple
$(\cS, \mathbf L^2(\R, dx), \cS' )$ where $\cS=\cS(\R)$ is the Schwartz space of smooth functions and $\cS' = \cS'(\R)$ is the space of tempered distributions.\\

However, in general elements of the $ML$ class do not have the product property so that the probability measure $P_{\varphi}$ is not a product measure. A typical example is the Mittag-Leffler function, corresponding to the $1-$dimensional grey noise probability space.\\

We now precise the domain of convergence of the product \eqref{fockphi}:

\begin{proposition}
The infinite product \eqref{fockphi} converges in the domain
\begin{equation}
\Omega_\varphi=\left\{(z_1,z_2,\ldots)\,\, :\sum_{j=1}^\infty |1-\varphi(|z_j|^2)| <\infty\right\}
\end{equation}
and contains in particular the space $\ell^2(\mathbb N,\mathbb C)$.
\label{corona123}
\end{proposition}

\begin{proof}
The function $\varphi(z\overline{w})$ is positive definite and thus it holds that
\begin{equation}
\label{sauce-bechaeml}
|\varphi(z\overline{w})|^2\le\varphi(|z|^2)  \varphi(|w|^2),\quad z,w\in\mathbb C.
\end{equation}
Since $\varphi(0)=1$ the first claim follows from the characterization of the convergence of infinite products. The second claim follows from the bound
\[
|\varphi(z)-1|\le K|z|,\quad |z|\le 1
\]
for some $K>0$.
\end{proof}

\subsection{A family of probability spaces}
Let us recall the version of the Bochner-Minlos theorem which we will be using here. The notation is that of \eqref{montmartre}.
\begin{theorem}
Let $f$ be a continuous complex-valued function on the Fr\'echet space $\cS(\R)$, such that $f(0)=1$, and assume that the kernel $f(s_1-s_2)$ is positive definite on $\cS(\R)$.
Then there exists a uniquely defined probability measure $P_f$ on the cylindrical sigma-algebra $\mathcal C$ of $\mathcal S^\prime(\R)$ such that
\[
f(s)=\int_{\cS^\prime(\R)}e^{i\langle \w,s\rangle}dP_f(\w).
\]
\end{theorem}

For a proof, see e.g. \cite{MR1408433,MR751959}.\\

Hence:

\begin{theorem}
Let $\varphi\in ML$.  There exists a uniquely defined probability measure $P_{\varphi}$, such that
\begin{equation}
\varphi \left(- \frac{\| s \|^2}{2}\right)  = \int_{\mathcal S'({\mathbb R})}e^{i \inner{\omega,s}}dP_{\varphi}(\omega), \quad s \in \cS(\R).
\label{iso1}
\end{equation}
\end{theorem}

In the proof of the Bochner-Minlos theorem one identifies $\mathcal S^\prime(\R)$ with a subspace of $\mathbb R^{\mathbb N}$. In the following proposition
we give the expression of the restriction of $P_\varphi$ on $\mathbb R^N$, $N\in\mathbb N$. In the sequel, $\hat{f}=\mathcal F(f)$
denotes the Fourier transform of $f$:
\begin{equation}
\label{ft}
\widehat{f}(u)=\int_{\mathbb R}e^{-iux}f(x)dx
\end{equation}

\begin{lemma} \label{Lem:FiniteMeasures} Let $\varphi \in ML.$ For every $N\in \N$ let $\{\xi_1, \ldots, \xi_N \} \subset \cS(\R)$ be an orthonormal set in $\mathbf L^2(\R, dx).$ Then, there
exists a normalized measure on $\R^N$
\begin{equation}
d\mu_{N} (x) = (2\pi)^{-N/2} \cF[{\varphi(-\|  \cdot \|^2/2)}] (x) dx_1 \cdots dx_N,
\end{equation}
such that the random variable
\begin{equation}
\omega \mapsto (Q_{\xi_1}(\omega), \ldots, Q_{\xi_N}(\omega)=(\langle \omega,\xi_1\rangle, \ldots, \langle \omega,\xi_N \rangle)  \in \R^N
\end{equation} has distribution 
$\mu_N$, that is
\begin{equation}
\E_\varphi[f (\inner{\omega, \xi_1}, \ldots, \inner{\omega, \xi_N})]  := (2\pi)^{-N/2} \int_{\R^N} f (x)  d\mu_{N}(x), 
\end{equation} for all 
$f \in \mathbf L^1(\R^N, d\mu_{N}),$ and where $\E_\varphi$ is the mathematical expectation with respect to $P_\varphi$.
\end{lemma}

\begin{proof} We follow arguments from \cite{MR1408433}. Let $f \in C^{\infty}_0(\R^N).$ Then,
\begin{gather*}
f (\inner{\cdot, \xi_1}, \ldots, \inner{\cdot, \xi_N}) = (2\pi)^{-N/2} \int_{\R^N} \hat f (y) e^{i ((\inner{\cdot, \xi_1}, \ldots, \inner{\cdot, \xi_N})| (y_1, \ldots, y_N ))} dy \\
= (2\pi)^{-N/2} \int_{\R^N} \hat f (y) e^{i \inner{\cdot, \sum_{j=1}^N y_j \xi_j }} dy,
\end{gather*}
where $(\cdot\,|\, \cdot)$ denotes the usual inner product in $\R^N$.
Hence,
\begin{eqnarray*}
\E_\varphi[f (\inner{\cdot, \xi_1}, \ldots, \inner{\cdot, \xi_N})] &=& (2\pi)^{-N/2} \int_{\R^N} \hat f (y) \E_\varphi [e^{i \inner{\cdot, \sum_{j=1}^N y_j \xi_j}}] dy \\
&=& (2\pi)^{-N/2} \int_{\R^N} \hat f (y) \varphi(-\| y \|^2/2) dy \\
&=& (2\pi)^{-N} \int_{\R^N} \left( \int_{\R^N} f (x) e^{-i\inner{x,y}} dx \right) \varphi(-\| y \|^2/2) dy \\
&=& (2\pi)^{-N/2} \int_{\R^N} f (x)  \underbrace{(2\pi)^{-N/2} \left( \int_{\R^N} e^{-i\inner{x,y}} \varphi(-\| y \|^2/2) dy \right) dx}_{= d\mu_{N}(x)} \\  
&=& (2\pi)^{-N/2} \int_{\R^N} f (x) \cF[{\varphi(-\| \cdot \|^2/2)}] (x) dx
\end{eqnarray*}

A density argument carries the result from $C^{\infty}_0(\R^N)$ to $\mathbf L^1(\R^N, d\mu_N).$ 
\end{proof}

We remark that 
$d\mu_N$ in the above arguments is indeed a positive measure since it is the Fourier transform of a positive definite function.
\section{The moments}
\setcounter{equation}{0}
\label{sec3}
While the Bochner-Minlos theorem provides us with the probability measure $P_\varphi$ for the study of stochastic processes and makes the link with the corresponding analytic Fock space we need the
series expansion (also known as Wiener-It\^o expansion) and this means we need to obtain the corresponding orthogonal polynomials. To this end we are going to study the corresponding moment problem which naturally starts with the computation of the various moments associated to $P_\varphi$.

Since our function $\varphi$ is entire we can consider its Laplace transform~\cite{ADKS1996}
$$
L(s)=\int_{\cS'(\R)}e^{\inner{\omega, s}}dP_\varphi(\omega), \quad s \in \cS(\R),
$$
and the $P_\varphi$-exponentials $e(s, \omega)=\frac{e^{\inner{\omega, s}}}{L(s)}$. This set is a total set in $\mathbf L^2(\mathcal S^\prime(\R),\mathcal C, dP_\varphi)$. Moreover, we also have the existence of all moments~\cite{ADKS1996}. In the special case that $P_\varphi$ is a Mittag-Leffler measure the existence of the Laplace transform was proven in~\cite{MR3315581}.

\subsection{Moment problem} Given $\varphi \in ML$, the moments can be computed from
$$1-\varphi_1 \frac{\| s \|^2}{2} +\varphi_2 \frac{\| s \|^4}{4} + \cdots = \int_{\cS'(\R)} \left( 1+i \inner{\omega, s}- \frac{\inner{\omega, s}^2}{2!}
  - \cdots  \right) dP_\varphi(\omega).$$
Replacing $s$ by $ts$ with $t\in\mathbb R$ we can rewrite this equality as
\[
1-t^2\varphi_1 \frac{\| s \|^2}{2} +t^4\varphi_2 \frac{\| s \|^4}{4} + \cdots = \int_{\cS'(\R)} \left( 1+it \inner{\omega, s}- t^2\frac{\inner{\omega, s}^2}{2!}
  - \cdots  \right) dP_\varphi(\omega).
\]
It follows that
\begin{equation} \label{isometry}
\varphi_1\langle s_1,s_2\rangle_2=\langle Q_{s_1},Q_{s_2}\rangle_{P_\varphi},\quad s_1,s_2\in\mathcal S({\mathbb R}).
\end{equation}
where the first inner product is the one of $\mathbf L^2(\mathbb R,dx)$.\smallskip

\begin{remark}
\label{r567}
  {\rm
 Equality \eqref{isometry} implies that the map $s\mapsto\frac{1}{\sqrt{\varphi_1}} Q_s$ can be extended in an isometric way to the whole Lebesgue space $\mathbf L^2(\mathbb R,dx)$. We denote by the same
 symbol the resulting map. This isometry allows to adapt much of the analysis of the papers
 \cite{aal2,aal3,ajnfao} to the present non-Gaussian setting.
 The function
 \begin{equation}
   \label{qerty123}
 X^\varphi_t(\w)=\frac{1}{\sqrt{\varphi_1}}Q_{1_{[0,t]}}(\w),\quad t\in\mathbb R
 \end{equation}
 will be called the associated stochastic process. It corresponds to a construction of the Brownian motion when $\varphi(z)=e^z$. We note that
 \begin{equation}
 \mathbb E_\varphi(X^\varphi_t\overline{X^\varphi_s})=t\wedge s,
\end{equation}
but we remark that it will not be Gaussian in general.
This process and its derivative will be studied in Section \ref{stochas}.
  }
\end{remark}
In the Gaussian case, the first two moments are enough to get all the moments.
Here, explicit computations are needed to obtain all of them. Before starting these calculations we would like to point out that a formula for the moments is already given in \cite{ADKS1996}
$$\int_{\cS'(\R)} Q_{s_1}(\omega) \cdots Q_{s_r}(\omega) dP_\varphi(\omega) = \nabla_{s_1} \cdots \nabla_{s_r} L(s) |_{s=0}, $$
where $\nabla_s$ denotes the directional derivative with respect to $s.$ Since this calculation will involve the determination of the coefficients of series expansion of $\varphi,$ which we already know, we prefer a more direct approach using a direct comparison of their coefficients.

We now present the moments required for the computation of an orthogonal basis of $\mathbf L^2(\mathcal S'({\mathbb R}),\mathcal C,dP_\varphi)$.
In the following, $\un\gamma=(\gamma_1,\gamma_2,\ldots)$ denotes a multi-index in $\mathbb N$ and $\un s=(s_1,s_2,\ldots)$ denotes a sequence of $s_i\in\mathcal S(\R)$.

\begin{theorem}
The moments $\int_{\cS'(\R)}  Q_{\un s}^{\un \gamma}(\omega) d P_\varphi (\omega) = \int_{\cS'(\R)} \prod_j  Q_{s_j}^{\gamma_j}(\omega) d P_\varphi (\omega)$ are given by
\begin{gather}
\int_{\cS'(\R)}  Q_{\un s}^{\un \gamma}(\omega) d P_\varphi (\omega) = 0 \mbox{ if } | \un \gamma | \mbox{ is odd,} \label{Comparison5a}
\end{gather} and if $| \un \gamma |=2n$ is even, then
\begin{gather}
\int_{\cS'(\R)}  Q_{\un s}^{\un \gamma}(\omega) d P_\varphi (\omega) =   \frac{(2n)! \varphi_n  \left( \sum_{|\un{\un \beta} |=n} {n\choose{\un{\un \beta}}}   \prod_{1\leq i \leq  j}   2^{\beta^0_{i,j}(1-\delta_{i,j})} \inner{ s_{i}, s_{j}}_{2}^{\beta^0_{i,j}} \right)}{\sum_{|\un{\un \beta} |=n} {n\choose{\un{\un \beta}}}    \prod_{1\leq i \leq  j}   2^{\beta^0_{i,j}(1-\delta_{i,j})}},
\end{gather}  whereas $\un{\un \beta} = (\beta_{i,j}^0)$ is the solution of the Diophantine system
\begin{gather}
\gamma_j = 2\beta_{j,j} + \sum_{i \not= j} \beta_{i,j},\quad j=1,2, \ldots   \label{Diophantine2}
\end{gather}
\end{theorem}
\begin{proof}

Let us now replace $s$ by a linear combination $t_1s_1+t_2s_2+\ldots$ in the equation for the moments. Then we have
$$
\sum_{n=0}^\infty (-1)^n \varphi_n\left(\int_{\R}(\sum_i t_i s_i)^2 dx \right)^n  =\sum_{m=0}^\infty \frac{i^m}{m!}\int_{\mathcal S'(\R)}\left(\sum_j t_j Q_{s_j}\right)^m dP_\varphi(\omega).
$$
We have
\begin{gather}\label{rhs0}
\left(\sum_i t_i s_i\right)^2  = \left(t_1 s_1 + t_2s_2+t_3 s_3+ \cdots\right)^2 = \sum_{|\un \alpha |=2} {{2}\choose{\un \alpha}} (\un{ts})^{\un \alpha},
\end{gather} where $\un \alpha = (\alpha_1, \alpha_2, \ldots), \un{ts} = (t_1s_1, t_2 s_2, \ldots ), (\un{ts})^{\un \alpha} := \prod_{j=1}^\infty (t_j s_j)^{\alpha_j},$ and ${{2}\choose{\un \alpha}}:= \frac{2!}{\alpha_1 ! \alpha_2! \cdots},$
so that
\begin{gather}
\sum_{n=0}^\infty (-1)^n \varphi_n\left( \int_{\mathbb{R}}\left(\sum_i t_i s_i\right)^2 dx \right)^n  = \sum_{n=0}^\infty (-1)^n \varphi_n\left( \sum_{|\un \alpha |=2} {{2}\choose{\un \alpha}} \int_{\mathbb{R}}(\un{ts})^{\un \alpha}dx \right)^n \nonumber  \\
 = \sum_{n=0}^\infty (-1)^n \varphi_n\left( \sum_{|\un \alpha |=2} {{2}\choose{\un \alpha}}\int_{\mathbb{R}}  \prod_{j=1}^\infty t_j^{\alpha_j} s_j^{\alpha_j}dx \right)^n   \nonumber \\
  = \sum_{n=0}^\infty (-1)^n \varphi_n\left(   {{2}\choose{\un \alpha^1}} \int_{\mathbb{R}} \prod_{j=1}^\infty t_j^{\alpha^1_j} s_j^{\alpha^1_j}dx + {{2}\choose{\un \alpha^2}} \int_{\mathbb{R}} \prod_{j=1}^\infty t_j^{\alpha^2_j} s_j^{\alpha^2_j}dx + {{2}\choose{\un \alpha^3}} \int_{\mathbb{R}} \prod_{j=1}^\infty t_j^{\alpha^3_j}  s_j^{\alpha^3_j}dx +  \cdots\right)^n   \label{rhs1}
\end{gather} where we recall, the sequences $\un \alpha^k$ satisfy $|\un \alpha^k|=2, k=1,2, 3, \ldots.$
Now, we evaluate the factors $ {{2}\choose{\un \alpha^k}} \int_{\mathbb{R}} (\prod_{j=1}^\infty t_j^{\alpha^k_j}  s_j^{\alpha^k_j})dx.$ As the sequences $\un \alpha^k$ satisfy $|\un \alpha^k|=2$ we have only two possible types:
\begin{itemize}
\item Type 1: $\un \alpha^k = (0,\ldots, 0, 2, 0, \ldots)$ with $2$ at the $j_0$ position. In this case,
$${{2}\choose{\un \alpha^k}} \int_{\mathbb{R}} (\prod_{j=1}^\infty t_j^{\alpha^k_j} s_j^{\alpha^k_j} ) dx = t_{j_0}^2 \int_{\mathbb{R}}s_{j_0}^{2}dx = t_{j_0}^2 \inner{ s_{j_0}, s_{j_0}}_{2}; $$
\item Type 2: $\un \alpha^k = (0,\ldots, 0, 1, 0,\ldots, 0, 1, 0, \ldots)$ with the $1's$ at the $(j_1, j_2)$ position $(j_1 < j_2).$ In this case,
$${{2}\choose{\un \alpha^k}} \int_{\mathbb{R}} (\prod_{j=1}^\infty t_j^{\alpha^k_j} s_j^{\alpha^k_j}) dx = 2t_{j_1}t_{j_2} \int_{\mathbb{R}}s_{j_1} s_{j_2}dx = 2t_{j_1}t_{j_2} \inner{ s_{j_1}, s_{j_2}}_{2};$$
\end{itemize} so that  (\ref{rhs1}) becomes
\begin{gather}
\sum_{n=0}^\infty (-1)^n \varphi_n\left( \int_{\mathbb{R}}\left(\sum_i t_i s_i\right)^2 dx \right)^n \nonumber \\
= \sum_{n=0}^\infty (-1)^n \varphi_n\left(  \sum_{j_0=1}^\infty  t_{j_0}^2 \inner{ s_{j_0}, s_{j_0}}_{2} + 2 \sum_{1\leq j_1 < j_2}   t_{j_1}t_{j_2} \inner{ s_{j_1}, s_{j_2}}_{2} \right)^n  \nonumber \\
= \sum_{n=0}^\infty (-1)^n \varphi_n\left(   \sum_{1\leq j_1 \leq  j_2} 2^{(1-\delta_{j_1, j_2})}  t_{j_1}t_{j_2} \inner{ s_{j_1}, s_{j_2}}_{2} \right)^n.\label{rhs2}
\end{gather}

Hence, the left-hand side  is
\begin{gather}
\sum_{n=0}^\infty (-1)^n \varphi_n\left( \int_{\mathbb{R}}\left(\sum_i t_i s_i\right)^2 dx \right)^n \nonumber \\
= \sum_{n=0}^\infty (-1)^n \varphi_n \sum_{|\un \beta |=n} {n\choose{\un \beta}} \prod_{1\leq j_1 \leq  j_2} 2^{\beta_{j_1, j_2}(1-\delta_{j_1, j_2})}  (t_{j_1}t_{j_2})^{\beta_{j_1, j_2}}
\inner{ s_{j_1}, s_{j_2}}_{2}^{\beta_{j_1, j_2}} \label{rhs3}
\end{gather} where $\un \beta = (\beta_{1,1}, \beta_{1, 2}, \beta_{1, 3}, \ldots, \beta_{2,2}, \beta_{2, 3}, \ldots ).$ \\

For the right hand side we have to consider the cases where $m$ is even since the imaginary part has to be zero. Here we get
\begin{gather}
\sum_{m=0}^\infty \frac{i^m}{m!} \int_{\cS'(\R)}\left( \sum_{j=1}^\infty t_j Q_{s_j} (\omega) \right)^m d P_\varphi (\omega)
= \sum_{n=0}^\infty \frac{(-1)^n}{(2n)!} \int_{\cS'(\R)} \left( \sum_{j=1}^\infty t_j Q_{s_j} (\omega) \right)^{2n} d P_\varphi (\omega), \label{ImagPart=0}
\end{gather}
with the condition from the imaginary part
\begin{gather}
\sum_{n=0}^\infty \frac{(-1)^{n}}{(2n+1)!} \int_{\cS'(\R)} \left( \sum_{j=1}^\infty t_j Q_{s_j} (\omega) \right)^{2n+1} d P_\varphi (\omega) =0. \label{ImagPart=0B}
\end{gather}

From   (\ref{ImagPart=0B}) we obtain for  $\un \gamma = (\gamma_j)$ such that $|\un \gamma|$ is odd
\begin{gather}
\int_{\cS'(\R)}  Q_{\un s}^{\un \gamma}(\omega) d P_\varphi (\omega) = \int_{\cS'(\R)} \prod_j  Q_{s_j}^{\gamma_j}(\omega) d P_\varphi (\omega) = 0. \label{Comparison4}
\end{gather}

Now, for (\ref{ImagPart=0}) we proceed as before with
\begin{gather*}
 \int_{\cS'(\R)} \left( \sum_{j=1}^\infty t_j Q_{s_j} (\omega) \right)^{2n} d P_\varphi (\omega) = \int_{\cS'(\R)}\left[ \left( \sum_{j=1}^\infty t_j Q_{s_j} (\omega) \right)^2 \right]^{n} d P_\varphi (\omega) \\
=  \int_{\cS'(\R)} \left(   \sum_{1\leq j_1 \leq  j_2} 2^{(1-\delta_{j_1, j_2})}  t_{j_1}t_{j_2}  Q_{s_{j_1}} Q_{s_{j_2}}  (\omega) \right)^{n} d P_\varphi (\omega),
\end{gather*} so that we obtain
\begin{gather}
\sum_{n=0}^\infty \frac{(-1)^n}{(2n)!} \int_{\cS'(\R)} \left( \sum_{j=1}^\infty t_j Q_{s_j} (\omega) \right)^{2n} d P_\varphi (\omega)  \label{RHS} \\
=\sum_{n=0}^\infty \frac{(-1)^n}{(2n)!} \sum_{|\un \beta | = n} {{n}\choose{\un \beta}}   \int_{\cS'(\R)}  \prod_{1\leq j_1 \leq  j_2} 2^{\beta_{j_1, j_2}(1-\delta_{j_1, j_2})}  (t_{j_1}t_{j_2})^{\beta_{j_1, j_2}}  Q_{s_{j_1}}^{\beta_{j_1, j_2}}(\omega) Q_{s_{j_2}}^{\beta_{j_1, j_2}} (\omega) d P_\varphi (\omega), \nonumber
\end{gather} where again $\un \beta = (\beta_{1,1}, \beta_{1, 2}, \beta_{1, 3}, \ldots, \beta_{2,2}, \beta_{2, 3}, \ldots ).$ This expression allows for a direct comparison between (\ref{rhs3}) and (\ref{RHS}). Indeed, we get
\begin{gather}
(2n)! \varphi_n \sum_{|\un \beta |=n} {n\choose{\un \beta}}   \prod_{1\leq j_1 \leq  j_2}   2^{\beta_{j_1, j_2}(1-\delta_{j_1, j_2})} ( t_{j_1}t_{j_2} )^{\beta_{j_1, j_2}} \inner{ s_{j_1}, s_{j_2}}_{2}^{\beta_{j_1, j_2}} \nonumber \\
= \sum_{|\un \beta | = n} {{n}\choose{\un \beta}}    \int_{\cS'(\R)}   \prod_{1\leq j_1 \leq  j_2} 2^{\beta_{j_1, j_2}(1-\delta_{j_1, j_2})}  (t_{j_1}t_{j_2})^{\beta_{j_1, j_2}}  Q_{s_{j_1}}^{\beta_{j_1, j_2}}(\omega) Q_{s_{j_2}}^{\beta_{j_1, j_2}} (\omega) d P_\varphi (\omega), \label{Comparison1}
\end{gather} which leads to the identification between the terms $\prod_{1\leq j_1 \leq  j_2}    ( t_{j_1}t_{j_2} )^{\beta_{j_1, j_2}}$ with same exponent.

Now, we observe that we have $\beta_{i,j} = \beta_{j,i},$ so that
\begin{gather*}
\prod_{1\leq j_1 \leq  j_2}    ( t_{j_1}t_{j_2} )^{\beta_{j_1, j_2}} = \prod_{j=1}^\infty t_j^{\sum_{i<j} \beta_{i,j} + 2 \beta_{j,j} + \sum_{j<i} \beta_{j,i}} = \prod_{j=1}^\infty t_j^{(2 \beta_{j,j} + \sum_{i\not= j} \beta_{i,j}) },
\end{gather*} under the usual convention $\sum_{i<1} =0.$

Hence, the evaluation of
\begin{gather*}
\int_{\cS'(\R)}  Q_{\un s}^{\un \gamma}(\omega) d P_\varphi (\omega) = \int_{\cS'(\R)} \prod_j  Q_{s_j}^{\gamma_j}(\omega) d P_\varphi (\omega),
\end{gather*} when $|\un\gamma|$ is even, depends on the resolution of the linear Diophantine system
\begin{gather}\label{Diophantine}
\gamma_j = 2\beta_{j,j} + \sum_{i \not= j} \beta_{i,j},\quad j=1,2, \ldots
\end{gather}
 or
$$
\left\{ \begin{array}{rcl}
\gamma_1 & = &  2\beta_{1,1} + \beta_{1,2}+ \beta_{1,3}+ \cdots \\
\gamma_2 & = & \beta_{1,2} + 2\beta_{2,2}+ \beta_{2,3}+ \cdots \\
\gamma_3 & = & \beta_{1,3} + \beta_{2,3}+ 2\beta_{3,3}+ \cdots \\
 & \vdots &
\end{array} \right.
$$

Remark that in these conditions $\sum_j \gamma_j = \sum_j (2\beta_{j,j} + \sum_{i \not= j} \beta_{i,j}) = 2 \sum_{i,j}\beta_{i,j}.$ Moreover, system (\ref{Diophantine}) is always solvable although non-uniquely. Let us denote as $\un{\un \beta} = (\beta_{i, j}^0)$ the solution of (\ref{Diophantine}) for a given $\un \gamma$ such that $| \un \gamma| = 2n.$ Then (\ref{Comparison1}) becomes
\begin{gather}
(2n)! \varphi_n \sum_{|\un{\un \beta} |=n} {n\choose{\un{\un \beta}}}   \prod_{1\leq i \leq  j}   2^{\beta^0_{i,j}(1-\delta_{i,j})} ( t_{i}t_{j} )^{\beta^0_{i,j}} \inner{ s_{i}, s_{j}}_{2}^{\beta^0_{i,j}} \nonumber \\
= \sum_{|\un{\un \beta} |=n} {n\choose{\un{\un \beta}}}   \int_{\cS'(\R)}   \prod_{1\leq i \leq  j}   2^{\beta^0_{i,j}(1-\delta_{i,j})} ( t_{i}t_{j} )^{\beta^0_{i,j}}  Q_{s_{i}}^{\beta_{i,j}^0}(\omega) Q_{s_{j}}^{\beta_{i,j}^0} (\omega) d P_\varphi (\omega). \label{Comparison2}
\end{gather}
For the left-hand side we have
\begin{gather*}
(2n)! \varphi_n \sum_{|\un{\un \beta} |=n} {n\choose{\un{\un \beta}}}   \prod_{1\leq i \leq  j}   2^{\beta^0_{i,j}(1-\delta_{i,j})} ( t_{i}t_{j} )^{\beta^0_{i,j}} \inner{ s_{i}, s_{j}}_{2}^{\beta^0_{i,j}} = (2n)! \varphi_n  \left(  \sum_{|\un{\un \beta} |=n} {n\choose{\un{\un \beta}}}   \prod_{1\leq i \leq  j}   2^{\beta^0_{i,j}(1-\delta_{i,j})} \inner{ s_{i}, s_{j}}_{2}^{\beta^0_{i,j}} \right) \un t^{\un \gamma}
\end{gather*}
while for the right-hand side we get
\begin{gather*}
\sum_{|\un{\un \beta} |=n} {n\choose{\un{\un \beta}}}   \int_{\cS'(\R)}   \prod_{1\leq i \leq  j}   2^{\beta^0_{i,j}(1-\delta_{i,j})} ( t_{i}t_{j} )^{\beta^0_{i,j}}  Q_{s_{i}}^{\beta_{i,j}^0}(\omega) Q_{s_{j}}^{\beta_{i,j}^0} (\omega) d P_\varphi (\omega) \\
= \left( \sum_{|\un{\un \beta} |=n} {n\choose{\un{\un \beta}}}    \prod_{1\leq i \leq  j}   2^{\beta^0_{i,j}(1-\delta_{i,j})} \right) \un t^{\un \gamma} \int_{\cS'(\R)} Q_{\un s}^{\un \gamma} (\omega) d P_\varphi (\omega).
\end{gather*}
Hence, we obtain
\begin{gather}
 \int_{\cS'(\R)} Q_{\un s}^{\un \gamma} (\omega) d P_\varphi (\omega) = \frac{(2n)! \varphi_n  \left( \sum_{|\un{\un \beta} |=n} {n\choose{\un{\un \beta}}}   \prod_{1\leq i \leq  j}   2^{\beta^0_{i,j}(1-\delta_{i,j})} \inner{ s_{i}, s_{j}}_{2}^{\beta^0_{i,j}} \right)}{\sum_{|\un{\un \beta} |=n} {n\choose{\un{\un \beta}}}    \prod_{1\leq i \leq  j}   2^{\beta^0_{i,j}(1-\delta_{i,j})}}, \label{Comparison3a}
\end{gather}
as desired.
\end{proof}

For example, let us compute $$ \int_{\cS'(\R)}   Q_{s_{1}}^{2}(\omega) Q_{s_{2}}^{2} (\omega) d P_\varphi (\omega),$$ that is, $\gamma_1=\gamma_2=2,$ and $\gamma_j =0, ~j\geq 3.$ The Diophantine system becomes
$$
\left\{ \begin{array}{rcl}
2 & = &  2\beta_{1,1} + \beta_{1,2}+ \beta_{1,3}+ \cdots \\
2 & = & \beta_{1,2} + 2\beta_{2,2}+ \beta_{2,3}+ \cdots \\
0 & = & \beta_{1,j} + \beta_{2,j}+ 2\beta_{j,j}+ \cdots , \quad j\geq 3,
\end{array} \right.
$$ and which has the (only) two obvious solutions $$\beta_{1,1}=\beta_{2,2}=1, \quad \mbox{other } \beta_{i,j}=0,$$ or
$$\beta_{1,2}=2, \quad \mbox{other } \beta_{i,j}=0.$$
The numerator is (recall, $n=2$)
\begin{gather*}
(2n)! \varphi_n  \left( \sum_{|\un{\un \beta} |=n} {n\choose{\un{\un \beta}}}   \prod_{1\leq i \leq  j}   2^{\beta^0_{i,j}(1-\delta_{i,j})} \inner{ s_{i}, s_{j}}_{2}^{\beta^0_{i,j}} \right) \\
= 4! \varphi_2  \left(  {2\choose{1,1}}  \inner{ s_{1}, s_{1}}_{2}\inner{ s_{2}, s_{2}}_{2} +{2\choose{2}}  2^2 \inner{ s_{1}, s_{2}}_{2} \right)\\
= 4! \varphi_2  \left(  2 \inner{ s_{1}, s_{1}}_{2}\inner{ s_{2}, s_{2}}_{L_2} + 4\inner{ s_{1}, s_{2}}_{2} \right),
\end{gather*} whereas the denominator is
\begin{gather*}
 \left( \sum_{|\un{\un \beta} |=n} {n\choose{\un{\un \beta}}}   \prod_{1\leq i \leq  j}   2^{\beta^0_{i,j}(1-\delta_{i,j})}  \right) =  \left(  {2\choose{1,1}}   +{2\choose{2}}  2^2  \right) = 6.
\end{gather*} This leads to
\begin{gather*}
 \int_{\cS'(\R)}   Q_{s_{1}}^{2}(\omega) Q_{s_{2}}^{2} (\omega) d P_\varphi (\omega) =\frac{4! \varphi_2}{6}   \left[ 2 \inner{s_1, s_1}_{2}\inner{s_2, s_2}_{2}+   4 \inner{ s_1, s_2}_{2}^2\right].
\end{gather*}

In particular, for $j=j_1 = j_2$ and $\beta_{j,j}= n$ one obtains
\begin{gather}
\int_{\cS'(\R)}  Q_{s_{j}}^{2n}(\omega) d P_\varphi (\omega) = (2n)! \varphi_n  \inner{s_{j}, s_{j}}_{2}^{n}. \label{Comparison3}
\end{gather}

\begin{corol}
If $\{ s_j, j \in \R \}$ is an o.n. basis for $\mathbf L^2(\R, dx)$ the following statements hold:
\begin{enumerate}
\item $\inner{Q_{s_i}, Q_{s_j}}_{P_\varphi} =0$ for all $i \not= j;$
\item $\int_{\cS'(\R)}  Q_{\un s}^{2\un \beta}(\omega) d P_\varphi (\omega) = (2n)! \varphi_n  \prod_{j=1}^\infty \inner{s_j, s_j}^{\beta_j}_{2},$ where $\un \beta = (\beta_1, \beta_2, \ldots)$ satisfy $|\un \beta|= n.$
\end{enumerate}
\end{corol}
\begin{proof} Statement \textit{(1)} is a direct consequence of (\ref{isometry}).

For \textit{(2)} we remark that all solutions $\un{\un \beta}$ of the Diophantine system (\ref{Diophantine}) in which $\beta_{i,j}^0 \not= 0$ for some $i\not=j$ corresponds to terms $\inner{s_i, s_j}_{2}=0.$ Thus only solutions $\un{\un \beta}$ in which $\beta^0_{i,j} = 0$ for all $i\not=j$ are admissible. This implies a unique solution $\un{\un \beta}$ for the system, $\gamma_j = 2\beta^0_{j,j} := 2\beta_j, ~ j=1,2,\ldots,$ and, for $n= |\un \beta|,$ we obtain
\begin{gather*}
\int_{\cS'(\R)}  Q_{\un s}^{2\un \beta}(\omega) d P_\varphi (\omega) = \int_{\cS'(\R)} \prod_{j=1}^\infty  Q_{s_j}^{2\beta_j}(\omega) d P_\varphi (\omega)  = (2n)! \varphi_n  \frac{ {n\choose{\beta}} \prod_{j=1}^\infty \inner{ s_{j}, s_{j}}_{2}^{\beta_j}}{ {n\choose{\beta}}}  \\
= (2n)! \varphi_n    \prod_{j=1}^\infty \inner{ s_{j}, s_{j}}_{2}^{\beta_j}.
\end{gather*}
\end{proof}

\begin{example}
As an example, the classic Gaussian case ($\varphi_n=\frac{1}{n!2^n}$ which corresponds to $\varphi(z)=e^z$) gives us

$$\int_{\cS'(\R)}  Q_{s}^{\gamma}(\omega)dP_\varphi(\omega)=\left\{
\begin{array}{cc}
0, & \gamma=2n+1\\
& \\
\frac{(2n)!}{n! 2^n}\| s \|^{\gamma/2},
& \gamma= 2n
\end{array}
\right..$$
The value for the even terms can be rewritten as $\frac{(2n)!}{n!2^n}= (2n-1)!!$ so that we get the moments of the Gaussian measure which generate the Hermite polynomials as orthogonal polynomials.
\end{example}

In~\cite{ADKS1996} the authors use the so-called $P_\varphi$-exponentials to construct polynomials which form a total set. In our case these $P_\varphi$-exponentials have the expression
\begin{eqnarray*}
  e(s,\omega) &=& \frac{e^{\inner{\omega, s}}}{L(s)} \\
  &=& \sum_{n=0}^\infty \sum_{k=0}^{\floor{n/2}}\frac{1}{(n-k)!}\left(\sum_{\alpha_0,\alpha_2,\alpha_4,\ldots: \sum i\alpha_{2i}=k}(-1)^{\sum_i \alpha_{2i}}\binom{\sum_i\alpha_{2i}}{\alpha_2,\alpha_4,\ldots}\Pi_i\varphi_i^{\alpha_{2i}}\right){\inner{s, s}_2^{k}} Q_{s}^{{n-2k} }(\omega).
\end{eqnarray*}
In particular, we can get the expansion
$$
e(ts,\omega)=\sum_{n=0}^\infty \frac{t^n}{n!} \sum_{k=0}^{\floor{n/2}}\frac{n!}{(n-k)!}\left(\sum_{\alpha_0,\alpha_2,\alpha_4,\ldots: \sum i\alpha_{2i}=k}(-1)^{\sum_i \alpha_{2i}}\binom{\sum_i\alpha_{2i}}{\alpha_0,\alpha_2,\alpha_4,\ldots}\Pi_i\varphi_i^{\alpha_{2i}}\right){\inner{s, s}}_2^{k}Q_{s}^{{n-2k} }(\omega)
$$
where $s\in \cS(\R), \omega\in \cS'(\R)$. Like in~\cite{ADKS1996} this allows us to introduce the polynomials
\begin{eqnarray*}
P(s_1,\ldots,s_n;\omega) & = & \sum_{k=0}^{\floor{n/2}}\frac{1}{(n-k)!}\sum_{(j_1,\ldots,j_n)}\left(\sum_{\alpha_2,\alpha_4,\ldots: \sum i\alpha_{2i}=k}(-1)^{\sum_i \alpha_{2i}}\binom{\sum_i\alpha_{2i}}{\alpha_0,\alpha_2,\alpha_4,\ldots}\Pi_i\varphi_i^{\alpha_{2i}}\right)\\ & & \quad \times { \inner{ s_{j_1} , s_{j_2} }}_2 \cdots {\inner{ s_{j_{2k-1}}, s_{j_{2k}} }_2 }Q_{s_{j_{2k+1}}}(\omega)\cdots Q_{s_{j_n}}(\omega),
\end{eqnarray*} where $\binom{n}{n_0,n_2,n_4,\ldots} = \frac{n!}{n_0! n_2 ! n_4! \cdots}.$
These polynomials satisfy the formulae (straightforward calculation)
$$
P(s_1,\ldots,s_n;\omega)=\sum_{k=1}^n\sum_{(j_1,\ldots,j_n)} Q_{s_{j_1}}(\omega)\cdots Q_{s_{j_k}}(\omega)P(s_{j_{k+1}},\ldots,s_{j_n},0)
$$
and
$$
Q_{s_1}(\omega)\cdots Q_{s_{n}}(\omega)=\sum_{k=1}^n\sum_{(j_1,\ldots,j_n)} M_k(s_{j_1},\ldots,s_{j_k})P(s_{j_{k+1}},\ldots,s_{j_n};\omega),
$$
where $M_k$ denotes the above calculated moments generated by $Q_{\un s}^{\un \gamma}(\omega)=Q_{s_{j_1}}(\omega) \cdots Q_{s_{j_k}}(\omega)$. In the same way as in~\cite{ADKS1996} we can get now the polynomials $P_n(\omega)$ by tensoring up and symmetrization. If we denote by
$s^{(n)}=\sum c_{j_1\ldots j_n}s_{j_1}\otimes\cdots\otimes s_{j_n}$ we can introduce the polynomials
$$
P_n(s^{(n)};\omega)=\sum c_{j_1\ldots j_n} P(s_{j_1},\ldots,s_{j_n};\omega)=\inner{s^{(n)},P_n(\omega)}.
$$
Furthermore, if we consider the symmetrization $sym(s^{(n)})=\frac{1}{n!}\sum s_{j_1}\otimes\cdots\otimes s_{j_n}$ we have $P_n(s^{(n)};\omega)=P_n(sym(s^{(n)});\omega)$. This allows us to consider the term $s^{\otimes n}=s\otimes\cdots\otimes s$ so that we have $P(s^{\otimes n};\omega)=\inner{s^{\otimes n}, P_n(\omega)}$. These polynomials now form a total set in $\mathbf L^2(\mathcal S^\prime(\R),\mathcal C, dP_\varphi)$.

Of course, this also allows us to give more explicit expressions for the so-called Appell system $\{P_n, Q_n\}$ constructed in~\cite{ADKS1996}.

\begin{example}
For the special case where the function $\varphi$ is given by the Mittag-Leffler function $E_\alpha$ we have for the $E_\alpha$-exponentials
$$
\sum_{n=0}^\infty \sum_{k=0}^{\floor{n/2}}\frac{1}{(n-k)!}\left(\sum_{\beta_0,\beta_2,\beta_4,\ldots: \sum i\beta_{2i}=k}(-1)^{\sum_i \beta_{2i}}\binom{\sum_i\beta_{2i}}{\beta_2,\beta_4,\ldots}\Pi_i\frac{1}{\Gamma(i\alpha+1) ^{\beta_{2i}}}\right){\inner{s, s}_2}^{k} Q_{s}^{{n-2k} }(\omega)
$$
as well as for the polynomials
\begin{eqnarray*}
P(s_1,\ldots,s_n;\omega) & = & \sum_{k=0}^{\floor{n/2}}\frac{1}{(n-k)!}\sum_{(j_1,\ldots,j_n)}\left(\sum_{\beta_2,\beta_4,\ldots: \sum i\beta_{2i}=k}(-1)^{\sum_i \alpha_{2i}}\binom{\sum_i\beta_{2i}}{\beta_0,\beta_2,\beta_4,\ldots}\Pi_i\frac{1}{\Gamma(i\alpha+1) ^{\beta_{2i}}}\right)\\ & & \quad \times { \inner{ s_{j_1} , s_{j_2} }}_2 \cdots {\inner{ s_{j_{2k-1}}, s_{j_{2k}} }_2 }Q_{s_{j_{2k+1}}}(\omega)\cdots Q_{s_{j_n}}(\omega)
\end{eqnarray*}
These are more explicit expressions than the ones given in~\cite{MR3315581}.
\end{example}

\subsection{Three term relations} The construction follows \cite{MR2581391}. For more information on orthogonal polynomials of several variables see also \cite{MR1827871}, and \cite{MR2130517}.

Let $\Pi^\infty_n$ denote the space of multivariate polynomials in infinite variables $\un z = (z_1, z_2, \ldots)$ with real valued coefficients,
of degree $n.$ Hence, $p \in \Pi^\infty_n$ can be written as $$p(\un z)=  \sum_{\un \gamma} p_{\un \gamma}~{\un z}^{\un \gamma} := \sum_{\un \gamma} p_{\un \gamma} \prod_{j=1}^\infty z_j^{\gamma_j},$$
and its degree is given as
$$\deg (p) := \max \{ |\un \gamma| : p_{\un \gamma} \not= 0 \}.$$

~

Now let us consider the space of polynomials of type $$P_{N} (\inner{\un s, \omega})=  \sum_{k=0}^N \left( \sum_{|\un \gamma|=k}  p_{\un \gamma} \inner{\un s, \omega}^{\un \gamma} \right) = \sum_{k=0}^N \left( \sum_{|\un \gamma|=k} p_{\un \gamma} \prod_{j=1}^\infty \inner{s_j, \omega}^{\gamma_j} \right),$$
in $ \mathbf L^2(\mathcal S^\prime({\mathbb R}),\mathcal C,dP_\varphi).$
For each $k,$ the space of homogeneous polynomials of degree $k$ is given by
$$\sum_{|\un \gamma|=k}  p_{\un \gamma} \inner{\un s, \omega}^{\un \gamma},$$
where $(p_{\un \gamma}) \in \ell^2$ (with possible exception of $k=0$). Still, we have for each $N$ that
$$\mathbf{P}_{N} = \{ P_{N}^1, P_{N}^2, \ldots,  P_{N}^l, \ldots\}, $$
is an ordered (infinite) basis for the set of polynomials of degree $N$ (hence, the dimension of $\Pi^\infty_N,$ denoted as $d(N),$ is $\infty$). This is to say,
$$\mbox{Span} {\mathbf{P}_{N}} = \mbox{Span} \{ P_{N}^1, P_{N}^2, \ldots,  P_{N}^l, \ldots\} = \Pi^\infty_N \cap (\Pi^\infty_{N-1})^\perp, \quad N \in \N_0.$$
From the construction it holds that  $\{ \mathbf{P}_0, \mathbf{P}_1, \cdots, \mathbf{P}_N \}$ is a  graded basis of $\Pi^\infty_N$ satisfying to
$$\inner{P_{N}^l, P_{N'}^{l'}}_{P_\varphi} =0, \quad \text{whenever } N \not= N'.$$
In a similar way, the polynomial
$$\inner{s_j, \omega} P_{N} (\inner{\un s, \omega}), \quad \text{where now } j \in \N,$$
is spanned by $\{ \mathbf{P}_0, \mathbf{P}_1, \ldots, \mathbf{P}_N,  \mathbf{P}_{N+1}\},$ that is,
$$\inner{s_j, \omega} P_{N} (\inner{\un s, \omega}) = \sum_{k=0}^{N+1} \sum_{l=1}^{\infty}  c^{j,l}_{N, k} ~  P_{k}^l(\inner{\un s, \omega}) =: \sum_{k=0}^{N+1} \mathbf{D}^j_{N,k} \mathbf{P}_k(\inner{\un s, \omega}),$$
where $\mathbf{D}^j_{N,k}$ denote operators acting on $\Pi^\infty_k \cap \mathbf L^2(\mathcal S^\prime(\R),\mathcal C,dP_\varphi).$


Hence we can express the product of $\inner{s_j, \omega}$ by the elements of the ordered basis $\mathbf{P}_{N}$ as
\begin{eqnarray}
\inner{s_j, \omega}\mathbf{P}_{N}(\inner{\un s, \omega}) 
& = & \sum_{k=0}^{N+1} \mathbf{D}^{j}_{N,k}  ~  \mathbf{P}_{k}(\inner{\un s, \omega})  \nonumber \\
& = &  \left( \sum_{k=0}^{N+1} ( c^{j}_{N,k} )^T_m ~  \mathbf{P}_{k}(\inner{\un s, \omega})\right)_{m=1}^{\infty}
\end{eqnarray}

~

We observe that
 \begin{gather}
  \Inner{\inner{s_j, \omega}\mathbf{P}_{N}, \mathbf{P}_{k}}_{P_\varphi} = \Inner{\mathbf{P}_{N}, \inner{s_j, \omega} \mathbf{P}_{k}}_{P_\varphi} =0,  \label{InnerTermsmallerN-2}
 \end{gather} for  $k=0,1, \ldots, N-2,$ which leads to
\begin{gather}
 \Inner{\inner{s_j, \omega}\mathbf{P}_{N}, \mathbf{P}_{k}}_{P_\varphi} = \mathbf{D}^{j}_{N,k} \Inner{\mathbf{P}_{k}, \mathbf{P}_{k}}_{P_\varphi}, \quad \text{for } k= N-1,N, N+1,  \label{InnerTermN-1NN+1}
 \end{gather}  where we remark that $\mathbf{D}^{j}_{N,k}$ denotes (by abuse of language) both the operator and the infinite dimensional matrix.

 From these relations we obtain the \textit{three terms recurrence relation}
\begin{gather}
\inner{s_j, \omega}\mathbf{P}_{N} (\inner{\un s, \omega})= \mathbf{A}_N^j \mathbf{P}_{N+1} (\inner{\un s, \omega}) + \mathbf{B}_N^j \mathbf{P}_{N} (\inner{\un s, \omega}) + \mathbf{C}_N^j \mathbf{P}_{N-1} (\inner{\un s, \omega}) \label{3TermsRecurrence}
 \end{gather} for $j, N \in \N_0$ under the usual convention of $\mathbf{P}_{-1} \equiv 0.$ Here, we have that the operators  are given by
\begin{gather}
\mathbf{A}_N^j  := \mathbf{D}^{j}_{N,N+1} : \Pi^\infty_{N+1} \cap \mathbf L^2(\mathcal S^\prime(\R),\mathcal C,dP_\varphi) \rightarrow \Pi^\infty_{N+1} \cap \mathbf L^2(\mathcal S^\prime(\R),\mathcal C,dP_\varphi), \nonumber \\
\mathbf{B}_N^j  := \mathbf{D}^{j}_{N,N} \in  \Pi^\infty_{N} \cap \mathbf L^2(\mathcal S^\prime(\R),\mathcal C,dP_\varphi) \rightarrow \Pi^\infty_{N+1} \cap \mathbf L^2(\mathcal S^\prime(\R),\mathcal C,dP_\varphi), \label{matricesin3termRecurrence} \\
\mathbf{C}_N^j  := \mathbf{D}^{j}_{N,N-1} \in  \Pi^\infty_{N-1} \cap \mathbf L^2(\mathcal S^\prime(\R),\mathcal C,dP_\varphi) \rightarrow \Pi^\infty_{N+1} \cap \mathbf L^2(\mathcal S^\prime(\R),\mathcal C,dP_\varphi). \nonumber
 \end{gather}
By Favard's Theorem the sequence $(\mathbf{P}_{N})_{N \in \N_0}$ is orthogonal. Furthermore, we have $\mathbf{C}_N^j = (\mathbf{A}_{N-1}^j)^\ast.$ For an explicit representation of $\mathbf{P}_N$ we can multiply (\ref{3TermsRecurrence}) on the left by the pseudo-inverse of $\mathbf{A}_N^j $ and perform a convenient rearrangement of the terms.


\subsection{Orthogonalization procedure} 

Since we have the moments for the basic polynomials $Q_{\un s}^{2\un \beta}$ to obtain an orthonormal basis we can proceed in the following iterative way:
\begin{enumerate}
\item We start with the basic polynomials associated to $\un \gamma = (\gamma_1, 0,0, \ldots) \sim \gamma_1,$ that is,
$$ Q_{\un s}^{\un \gamma}(\omega) =  Q_{(s_1, 0, \cdots)}^{(\gamma_1, 0, \ldots)}(\omega), \quad  \omega \in \cS'(\R), s_1 \in \cS(\R), \gamma_1 \in \N_0.$$
The set of all these polynomials forms an orthogonal family that we now normalize
$$\{ \tilde Q^{\gamma_1}, \quad  \gamma_1 \in \N_0 \}$$
and we denote by $\cH_1$ its span 
 \begin{equation} \label{on family 1}
 \cH_1 = \mbox{Span} \{ \tilde Q^{\gamma_1}, \quad \gamma_1 \in \N_0 \}.
 \end{equation}

\item We consider now the basic polynomials associated to $\un \gamma = (\gamma_1, \gamma_2, 0, \ldots) \sim (\gamma_1, \gamma_2),$ that is,
$$
 Q_{\un s}^{\un \gamma}(\omega) =   Q^{(\gamma_1, \gamma_2)}(\omega), \quad  \omega \in \cS'(\R), \gamma_1, \gamma_2  \in \N_0.
 $$ Clearly, $\cH_1$ is a subset of $\mbox{Span} \{ Q^{(\gamma_1, \gamma_2)}, \gamma_1, \gamma_2  \in \N_0 \}$. Hence, we can decompose
 this space into the following orthogonal sum:
$$\mbox{Span} \{ Q^{(\gamma_1, \gamma_2)}, \gamma_1, \gamma_2  \in \N_0 \} = \cH_1 \oplus \cH_1^{\bot},$$
and use again Gram-Schmidt orthogonalization on
$$\mbox{proj}_{\cH_1^\bot}   \mbox{Span} \{ Q^{(\gamma_1, \gamma_2)} \}, $$
to construct the orthonormal family
$$\{ \tilde Q^{(\gamma_1, \gamma_2)},  \gamma_1 \in \N_0, \gamma_2  \in \N \}$$ again under the restriction
$$\| \tilde Q^{(\gamma_1, \gamma_2)} \|_{P_\varphi} =1.$$

We remark, at this point, that this orthogonalization requirement is automatically fulfilled in the case where the measure $P_\varphi$ is a product measure.

Again, we denote by
 \begin{equation} \label{on family 2}
 \cH_2 = \mbox{Span} \{ \tilde Q^{(\gamma_1, \gamma_2)},  \gamma_1 \in \N_0, \gamma_2  \in \N \},
 \end{equation} where it holds
 $$\mbox{Span} \{ \tilde Q^{(\gamma_1, \gamma_2)}, \gamma_1, \gamma_2  \in \N_0 \} =
 \cH_1 \oplus \cH_2,$$ with the elements of $\cH_1$ being orthogonal to the ones of $\cH_2.$\\

 Also, $\cH_1 \oplus \cH_2$ is a subset of the span of the basic polynomials associated to $\un \gamma = (\gamma_1, \gamma_2, \gamma_3, 0, \ldots) \sim
 (\gamma_1, \gamma_2, \gamma_3),$ that is,
 $$\cH_1 \oplus \cH_2 \subset \mbox{Span} \{ \tilde Q_{\un s}^{\un \gamma} =
 \tilde Q^{(\gamma_1, \gamma_2, \gamma_3)},  \gamma_1, \gamma_2, \gamma_3  \in \N_0 \},$$ again subjected to
$$\| \tilde Q^{(\gamma_1, \gamma_2, \gamma_3)} \|_{P_\varphi} = 1.$$

Hence, the following decomposition still holds:
$$\mbox{Span} \{ \tilde Q^{(\gamma_1, \gamma_2, \gamma_3)}, \gamma_1, \gamma_2, \gamma_3  \in \N_0 \} =
(\cH_1 \oplus \cH_2) \oplus (\cH_1 \oplus \cH_2)^\bot.$$
\item By iteration over $N =2, 3, \cdots$, for each
 $$\mbox{Span} \{  Q^{(\gamma_1, \ldots, \gamma_{N+1})}, \gamma_i  \in \N_0, i=1, \ldots, N+1 \} = (\oplus_{j=1}^N \cH_j) \oplus (\oplus_{j=1}^N \cH_j)^\bot,$$
 we construct the orthogonal family
 $$\{ \tilde Q^{(\gamma_1, \cdots, \gamma_{N+1})}, \gamma_i  \in \N_0, i=1, \ldots, N  \in \N_0, \gamma_{N+1} \in \N \}$$
 by Gram-Schmidt orthogonalization and subsequent norm restriction, on
$$\mbox{proj}_{(\oplus_{j=1}^N \cH_j)^\bot}   \mbox{Span} \{  Q^{(\gamma_1, \cdots, \gamma_{N+1})}, \gamma_i  \in \N_0, i=1, \cdots, N+1 \},$$ and denote the resulting space by
 \begin{equation} \label{on family N+1}
\cH_{N+1} = \mbox{Span}\{ \tilde Q^{(\gamma_1, \ldots, \gamma_{N+1})}, \gamma_i  \in \N_0, i=1, \ldots, N  \in \N_0, \gamma_{N+1} \in \N \}.
 \end{equation}

 We remark that
\begin{itemize}
\item $\oplus_{j=1}^{N+1} \cH_j \subset \mbox{Span} \{ \tilde Q^{(\gamma_1, \ldots, \gamma_{N+2})},  \gamma_i  \in \N_0, i=1,
\ldots, N+2 \};$
\item $\mbox{Span} \{ \tilde Q^{(\gamma_1, \ldots, \gamma_{N+2})}, \gamma_i  \in \N_0, i=1, \ldots, N+2 \} = (\oplus_{j=1}^{N+1} \cH_j)
\oplus (\oplus_{j=1}^{N+1} \cH_j)^\bot;$
 \end{itemize}
\end{enumerate}

From this, we obtain that
$$\mbox{Span} \{ \tilde Q^{\un \gamma}(\omega), \gamma_j  \in \N_0, j \in \N_0 \} = \oplus_{j=1}^\infty \cH_j.$$

The above constructed basis of orthogonal polynomials allows us to get the following theorem. In the theorem, and in the sequel we fix $s_j \in \cS(\R)$ to be the
Hermite functions $\zeta_j$ and write $Q^{\un \gamma}(\omega)$ without explicit mention of $\zeta_j$.

\begin{theorem} Every $f \in \mathbf L^2(\mathcal S^\prime(\R),\mathcal C,dP_\varphi)$ has a unique representation
$$f(\w) = \sum_{\un \gamma} c_{\un \gamma} \tilde Q^{\un \gamma}(\w),$$
with $c_{\un \gamma} \in \mathbb C$ for all $\un \gamma.$
\end{theorem}

Following the arguments in \cite[p. 723]{ajnfao} one sees that the  $\widetilde{Q}^{\un\gamma}$ are indeed an orthonormal basis of $\mathbf L^2(\cS^\prime(\R),\mathcal C,dP_\varphi)$. These results allow us to construct a Fock space linked to the $\mathbf L^2(\cS^\prime(\R),\mathcal C,dP_\varphi)$ space.

\section{Generalized Fock spaces}
\setcounter{equation}{0}
This section presents a result of independent interest, which will be used in Section \ref{fock} to build the counterpart of the Fock space for the space
$\mathbf L^2(\mathcal S^\prime(\R), \mathcal C, dP_\varphi)$. We are given a sequence of positive definite functions $K_j(z_j,w_j)$ of the form
\[
K_j(z_j,w_j)=1+\underbrace{\sum_{k=1}^\infty z_j^k\ov w_j^k \alpha_{k_j}}_{k_j(z_j\overline{w_j})}
\]
where the numbers $\alpha_{k_j}$ are assumed positive and where $z_j,w_j$ run through some neighborhood $\Omega_j$ of the origin.
We consider
\begin{equation}\label{Eq:3.001}
\Omega := \{ \un z = (z_1, z_2, \ldots) \in \C^{\N} : 0<\prod_{j=1}^\infty k_j(z_j, z_j) < \infty \},
\end{equation}
which always contains sequences with at most a finite number of non-zero entries. Denote by $K$ the positive definite kernel acting on $(\un z, \un w) \in \Omega \times \Omega$ as
 \begin{equation}\label{Eq:3.002}
K(\un z, \un w) :=  \prod_{j=1}^\infty (1+k_j(z_j, w_j) ), \quad \mbox{for all } (\un z, \un w) \in \Omega \times \Omega.
\end{equation}

We set
\begin{equation}\label{Eq:3.003}
K_N(\un z, \un w) :=  \prod_{j=1}^N (1+k_j(z_j, w_j) ), \quad \mbox{for all } (\un z, \un w) \in \Omega \times \Omega,
\end{equation}
we have
$$K(\un z, \un w)-K_N(\un z, \un w) = K_N(\un z, \un w) \left( \prod_{j=N+1}^\infty (1+k_j(z_j, w_j) ) \right) \geq 0,$$ as both kernels in the product are positive definite. hence,
$$K_N(\un z, \un w) \leq K(\un z, \un w), \quad \mbox{for all } (\un z, \un w) \in \Omega \times \Omega.$$

\begin{proposition}\label{Lem:3.001}
In the above notation we have the following decomposition:
\begin{gather*}
\cH(K) := \C \oplus \left( \oplus_j \cH(k_j)  \right) \oplus \left( \oplus_{j_1 < j_2} \cH(k_{j_1}) \otimes \cH(k_{j_1})  \right) \oplus \cdots \\
\cdots \oplus \left( \oplus_{j_1 < \cdots < j_N} \cH(k_{j_1}) \otimes \cdots \otimes \cH(k_{j_N})  \right) \oplus \cdots \end{gather*} where $\cH(k_j)$ denotes the reproducing space associated to the kernel $k_j(z_j, w_j)$.
\end{proposition}

\begin{proof} For simplicity, denote $k_j(z_j, w_j) = k_j,$ so that $K_N = (1+k_1) \cdots (1+k_N).$ Hence,
\begin{gather*}
\cH(K_1) := \C \oplus \cH(k_1),
\end{gather*}
while
\begin{gather*}
K_2 = (1+k_1)(1+k_2) = 1+ k_1 + k_2 + k_1 k_2,
\end{gather*}implies
\begin{gather*}
\cH(K_2) := \C \oplus \left(\cH(k_1) \oplus \cH(k_2) \right) \oplus \left( \cH(k_1) \otimes \cH(k_2) \right),
\end{gather*}
while for $N=3$
\begin{gather*}
K_3 = (1+k_1)(1+k_2)(1+k_3) = (1+ k_1 + k_2 + k_1 k_2)(1+k_3) \\
= 1+ k_1 + k_2 + k_3+ k_1 k_2 + k_1k_2 + k_2k_3 + k_1 k_2k_3,
\end{gather*} that is to say we get
\begin{gather*}
\cH(K_3) := \C \oplus \left(\cH(k_1) \oplus \cH(k_2) \oplus \cH(k_3) \right) \oplus \left( \oplus_{j_1 < j_2} \cH(k_{j_1}) \otimes \cH(k_{j_2}) \right) \\
\oplus \left( \cH(k_1) \otimes \cH(k_2) \otimes \cH(k_3) \right).
\end{gather*} Proceeding by induction we obtain the desired result. Furthermore, each $\cH(K_N)$ is embedded isometrically in $\cH(K).$ Indeed, we have
\begin{gather*}
K - K_N = K_N \left(  \prod_{j=N+1}^\infty (1+k_j )  -1 \right) := K_N R_N,
\end{gather*}  where $R_N := \prod_{j=N+1}^\infty (1+k_j )  -1.$ Hence, this implies $1 \notin \cH(R_N),$  $\cH(K) = \cH(K_N) \oplus \left( \cH(K_N) \otimes \cH(R_N)  \right),$ and $\cH(K_N) \cap \cH(R_N) = \{  0 \}.$

Assume $f \in \cH(K) \ominus \cH(R_N).$ Since
$$\prod_{j=1}^\infty (1+k_j ) = 1+ \sum_j k_j + \sum_{j_1 < j_2} k_{j_1} k_{j_2} + \cdots $$ we get that $f \rightarrow 0$ as $N \rightarrow \infty,$ thus ensuring convergence.
\end{proof}

\section{$\varphi$-transform and Fock spaces}
\label{fock}
\setcounter{equation}{0}
Let $\un \alpha=(\alpha_1,\ldots, \alpha_n)\in\mathbb N_0^n$ (in particular some of the $\alpha_j$ may be equal to $0$).
We set \begin{eqnarray}
\un z^{\un \alpha}&=&z_1^{\alpha_1}\cdots z_n^{\alpha_n}\\
\varphi_{\un \alpha}&=&\varphi_{\alpha_1}\cdots \varphi_{\alpha_n}\\
|\un \alpha|&=&\alpha_1+\cdots +\alpha_n.
\end{eqnarray}

Recall that we have fixed an orthonormal basis $\zeta_1,\zeta_2,\ldots$ of $\mathbf L^2(\mathbb R,dx)$, where $\zeta_j$'s are the Hermite functions,
and consider the associated orthonormal basis constructed in Section 3, that is, the basis spanned by $\tilde Q^{\un \gamma}.$


\begin{definition} For our basis we define the corresponding Wick product in the usual way via
\begin{equation}
\label{wickp}
\tilde Q^{ \un \gamma} \lozenge \tilde Q^{\un \delta}=\tilde Q^{\un \gamma+\un \delta}.
\end{equation}
\end{definition}

The Wick product is not a law of composition in $\mathbf L^2(\cS^\prime(\R), \mathcal C, dP_\varphi)$. We build in the next section a larger space, in which the Wick product is stable.

\begin{definition}
The $\varphi$-transform sends $\tilde Q^{\un \gamma}$ to $\dfrac{\un z^{\un \gamma}}{\sqrt{\varphi_{\un \gamma}}}.$
\end{definition}

\begin{theorem}
Under the $\varphi$-transform the space $\mathbf L_2(\cS^\prime(\R),\mathcal C,dP_\varphi)$ is mapped into the reproducing kernel Hilbert space of functions with
reproducing kernel equal to \eqref{fockphi}.
\end{theorem}
\begin{proof}
With the above multi-index notation we have
\[
\prod_{j=1}^\infty\varphi(z_j\overline{w_j})=1+\sum_{n=1}^\infty\sum_{|\un \alpha|=n}\varphi_{\un \alpha} \un z^{\un \alpha} \un{\overline{w}}^{\un \alpha}
\]
so that the Fock space is now equal to the set of power series of the form
\[
f(z)=\sum_{\un \alpha\in L} f_{\un \alpha} \un z^{\un \alpha}
\]
with (square of the) norm defined by
\[
\|f\|_\varphi^2=\sum_{\un \alpha\in L}\frac{|f_{\un \alpha}|^2}{\varphi_{\un \alpha}}.
\]
The $\varphi$-transform reads
\[
\sum_{\un \alpha\in L} f_{\un \alpha} \widetilde{Q}^{\un\alpha}\,\,\mapsto\,\, \sum_{\un \alpha\in L} f_{\un \alpha} \frac{\un z^{\un \alpha}}{\sqrt{\varphi_{\un \alpha}}}
\]
is then unitary.
\end{proof}

In the classical setting, the Fock space is characterized to be the unique space of power series in which the adjoint of the operator $M_z$ of
multiplication by $z$, is differentiation with respect the complex variable,
or, equivalently, the adjoint of differentiation with respect to the complex variable is the operator of multiplication by $z$.\smallskip

We now study the counterparts of these operators when $\varphi$ fulfills some supplementary hypothesis. More precisely,
if additionally the function $\varphi$ is an entire function with order $\rho >0$ and  degree $\sigma > 0,$ that is, such that $$\lim_{k \rightarrow \infty} k^{\frac{1}{\rho}} \sqrt[k]{|\varphi_k|} =\left(\sigma e \rho\right)^{\frac{1}{\rho}}$$ we can consider the Gelfond-Leontiev operator $\partial^{\varphi}$ of
generalized differentiation associated to $\varphi$ which acts on an analytic function $f(z) =\sum_{k=0}^\infty a_k z^k,  |z|<1,$ as
$$
f \mapsto \partial^\varphi f(z) =\sum_{k=1}^\infty \frac{\varphi_{k-1}}{\varphi_k} ~a_k z^{k-1}.$$
See \cite{MR0045812,MR1265940}.
Then, we have
\begin{equation}
\label{123321}
\partial^\varphi \varphi(z) = \varphi(z).
\end{equation}
When $\varphi(z)=e^z$ one had $\partial^\varphi=\partial$, and \eqref{123321} reduces to $\partial e^z=e^z$.

\begin{theorem}
Let $\mathcal H(\varphi)$ be the reproducing kernel Hilbert space with reproducing kernel $\varphi(z\overline{w})$.
Then,
\begin{eqnarray}
(\partial^\varphi)^*&=&M_z\\
M_z^*&=&\partial^\varphi.
\end{eqnarray}
\end{theorem}

\begin{proof} We just give the gist of the proof. Let $k,\ell\in\mathbb N_0$. We have
\begin{equation}
\begin{split}
\langle\partial^\varphi z^k,z^\ell\rangle&=\frac{\varphi_{k-1}}{\varphi_k}\langle z^{k-1},z^\ell\rangle\\
&=\begin{cases} \,\, 0\,\,\,\,\quad\hspace{1.95cm} \ell\not= k-1\\
  \frac{\varphi_{k-1}}{\varphi_k}\langle z^{k-1},z^{k-1}\rangle,\,\,\,\, k-1=\ell   \end{cases}\\
&=\begin{cases} \,\, 0\,\,\,\,\quad\hspace{1.95cm} \ell\not= k-1\\
                \frac{\varphi_{k-1}}{\varphi_k}\frac{1}{\varphi_{k-1}}=\frac{1}{\varphi_k},\,\,\,\, k-1=\ell   \end{cases}
\end{split}
\end{equation}
and
\begin{equation}
\begin{split}
\langle z^k,M_zz^\ell\rangle&=\langle z^{k},z^{\ell+1}\rangle\\
&=\begin{cases} \,\, 0\,\,\,\,\quad\hspace{1.95cm} \ell+1\not= k\\
  \langle z^{k},z^{k}\rangle,\,\,\,\, \hspace{1.2cm}k-1=\ell   \end{cases}\\
&=\begin{cases} \,\, 0\,\,\,\,\quad\hspace{1.95cm} \ell\not= k-1\\
                \frac{1}{\varphi_k},\,\,\,\,\hspace{2.0cm} k-1=\ell   \end{cases}
\end{split}
\end{equation}
\end{proof}

As an immediate consequence we have:
\begin{lemma}
Let $k,\ell\in\mathbb N_0$ It holds that:
\begin{enumerate}
\item $\langle\partial^\varphi z^k,z^\ell \rangle = \langle z^k,M_zz^\ell\rangle = \left\{  \begin{array}{ccc}
0,& \qquad& k-1\not= \ell \\
\varphi_k^{-1} ,& \qquad& k-1= \ell
\end{array} \right. ;$

~

\item $[\partial^\varphi, M_z] z^k = \big( \frac{\varphi_k^2 -\varphi_{k-1} \varphi_{k+1}}{\varphi_{k} \varphi_{k+1}}  \big)z^k, \quad k=0,1,2, \ldots,$
\end{enumerate} under the convention $\varphi_{-1}=0.$
\end{lemma}

We conclude this section by a computation which shows the counterpart of the formula $\partial e^{az}=ae^{az}$.

\begin{proposition}
Let $K_\varphi(\un z,\un w)$ be defined by \eqref{fockphi}, that is
\[
K_\varphi(\un z,\un w)=\prod_{j=1}^\infty \varphi(z_j\overline{w_j}),
\]
and let $\partial_j^\varphi$ denote the operator $\partial^\varphi$ applied to the variable $z_j$.
It holds that
\[
\partial^\varphi_j K_\varphi(\un z,\un w)=\overline{w_j}K_\varphi(\un z,\un w).
\]
\end{proposition}

\begin{proof}
We have
\begin{eqnarray*}
  \partial^\varphi_j K(\un z,\un w)
  & = & \partial^\varphi_j \left( \prod_{i=1}^\infty \varphi(z_i\overline{w_i}) \right) \nonumber \\
& = & \prod_{i=1, i\not=j}^\infty  \varphi(z_i\overline{w_i}) \left(  \partial^\varphi_j \varphi(z_j\overline{w_j}) \right)  \nonumber \\
& = & \prod_{i=1, i\not=j}^\infty  \varphi(z_i\overline{w_i}) \left(  \partial^\varphi_j \sum_{k=0}^\infty \varphi_k  (z_j\overline{w_j})^k \right)  \nonumber \\
& = & \prod_{i=1, i\not=j}^\infty  \varphi(z_i\overline{w_i}) \left(   \sum_{k=1}^\infty \frac{\varphi_{k-1}}{\varphi_k} \varphi_k  z_j^{k-1} \overline{w_j}^k \right)  \nonumber \\
&= &\overline{w_j}K_\varphi(\un z,\un w).  \label{Eq:5.6}
\end{eqnarray*}

\end{proof}

\section{Strong algebras}
\setcounter{equation}{0}
In the construction of the probability space, the Wick product and stochastic processes a number of problems arise. First the Wick product needs not be a law of composition. Next, given an
$\mathbf L^2(\cS'(\R), \mathcal C,dP_\varphi)$-valued process, it is not in general differentiable in the corresponding topology. One way to handle these questions is to embed the space $\mathbf L^2(\cS'(\R), \mathcal C,dP_\varphi)$
into a Gelfand triple $S \subset \mathbf L^2(\cS'(\R), \mathcal C, dP_\varphi)\subset S^\prime$, where $S$ is a space of stochastic  test functions and $S^\prime$ is a space of stochastic distributions. We require
that $S^\prime$ has an algebra structure, of a very special kind, first introduced on an example by Kondratiev and V\r{a}ge, and later formalized in the series of
papers \cite{MR3231624,vage1,MR3038506,MR3404695}.

\begin{definition} (see \cite[p. 211-212]{MR3404695})
Let $\mathcal A$ denote an algebra which is an inductive limit of a family of
Banach spaces $\left\{X_\alpha\,;\, \alpha\in A\right\}$ directed under inclusion. We call
$\mathcal A$ a strong algebra if for every $\alpha\in A$ there exists $h(\alpha)\in A$
such that, for every $\beta\ge h(\alpha)$, there is a positive constant
$A_{\beta,\alpha}$ such that
\begin{equation}
\|ab\|_\beta \le A_{\beta,\alpha}\|a\|_{\alpha}\cdot\|b\|_{\beta}\quad and\quad  \|ba\|_\beta \le A_{\beta,\alpha}\|a\|_{\alpha}\cdot\|b\|_{\beta}
\end{equation}
for every $a\in X_\alpha$ and $b\in X_\beta$.
\end{definition}

Strong algebras are topological algebras in the sense of \cite{MR0438123}.
Topological algebras are defined to have  a product separately continuous in each variable. For a strong algebra one has:

\begin{proposition} (see \cite[Theorem 3.3, p.215]{MR3404695})
In a strong algebra if any set is bounded if and only if it is bounded in one of the $X_\alpha$, then the product is jointly continuous in the two variables.
\label{rab}
\end{proposition}

\begin{proposition} (see \cite[Proposition 2 p 59, with Th\'eor\`eme 3. p. 49]{GS2})
In a nuclear strong algebra, a set is (weakly or strongly) compact if and only if it is closed and bounded, and it is then included, and compact, in one of the
spaces $\mathcal H_p$.
\label{mur}
\end{proposition}

Strong algebras include Banach algebras, but are really of interest in an ``orthogonal case'', where the algebra is the dual of a nuclear Fr\'echet space, and where
in particular, a set is compact if and only if it is closed and bounded (the latter of course never happens in infinite dimensional normed spaces).\\

To simplify the presentation we assume that $\mathcal A$ is an inductive limit of an increasing sequence of Hilbert spaces $\mathcal H_p$, with
decreasing norms $\|\cdot\|_p$, $p=0,1,\ldots$, and that the limit is nuclear (following \cite{GS2} one could assume the space perfect). The
space $\mathcal A$ is then in particular reflexive and not separable, but the following holds for sequences:

\begin{proposition} (see \cite[Th\'eor\`eme 4, p. 58]{GS2})
If a sequence converges (strongly or weakly) in $\mathcal A$, it is contained after a certain rank in one of the spaces $\mathcal H_p$ and converges
in the norm $\|\cdot\|_p$.
\end{proposition}

The previous result is very important to study the continuity of functions from a metric space into $\mathcal A$.\\

We now construct a family of strong algebra associated to $\cH(K)$. We take a sequence $\un a = (a_n)$ such that
  \begin{equation}
    a_n a_m \leq a_{n+m}
  \end{equation}
  and
  \begin{equation}
\sum_n a_n^{-d} < \infty \mbox{\text~for~some fixed integer~}d>1. \label{Eq:4.003}
  \end{equation}
Define now $\un a^{\un \alpha} := a_1^{\alpha_1} a_2^{\alpha_2} \cdots.$ Then we have for all $\un \alpha, \un \beta \in L,$ where $L$ is the set of multi-indexes $\un \alpha =(\alpha_1, \alpha_2, \ldots ),$
$$\un a^{\un \alpha} \un a^{\un \beta} \leq \un a^{\un \alpha+\un \beta}.$$
Given a second sequence $\un b = (b_n)$ satisfying also to $b_n b_m \leq b_{n+m}$, and satisfying
$$\sum_{n} b_n^{-d}  < \infty$$
where $d$ is as in \eqref{Eq:4.003}.\smallskip

We construct the weights $\un d_{n, \un \alpha}$ as
\begin{gather}
\un d_{n, \un \alpha} := b_{n} \un a^{\un \alpha},\quad{\rm for}\quad |\un\alpha|=n.
\label{Eq:4.004}
\end{gather}

\begin{definition}
Let $p=1,2,\ldots$.
We denote by $\mathcal H_p(a,b)$ the space of power series
\[
\sum_{n=0}^\infty\sum_{|\un\alpha|=n}f_{\un\alpha}\un z^{\un\alpha}
\]
such that
\begin{gather}
\| f \|_p := \sum_{n=0}^\infty b_n^{-p}\left(\sum_{|\un\alpha|=n}  |f_{\un \alpha}|^2 \un a_{\un\alpha}^{-p}\right)<\infty. \label{Eq:4.005}
\end{gather}
\end{definition}

We define the convolution
\[
f\ast g=\sum_{n=0}^\infty\left(\left(\sum_{\substack{\un \alpha + \un \beta=\un \gamma\\ |\un\gamma|=n}}f_{\un\alpha}g_{\un\beta}\right)\un z^{\un\gamma}\right).
\]

\begin{proposition}
\label{Lem:4.001}
Let $d$ be as in \eqref{Eq:4.003} and \eqref{Eq:4.004}.
Let $p,q\in\mathbb N$ be such that $p-q>d$, and let $f\in\mathcal H_q(a,b)$ and $g\in\mathcal H_p(a,b)$. We have
\begin{gather}
\| f \ast g \|_p \leq A_{p, q} \| f  \|_q \|  g \|_p, \label{Eq:4.0051}
\end{gather}
where
\[
A_{p,q}=\left( \sum_{\un \alpha} \un a_{\un \alpha}^{-(p-q)} \right) \left( \sum_{n=0}^\infty b_{n}^{-(p-q)}\right).
\]
\end{proposition}
\begin{proof}
From \cite[Proposition 2.2]{MR1408433}, we have that, if $a_1 >1$, then
$$\sum_{\un \alpha} \un a_{\un \alpha}^{-(p-q)} < \infty, \quad p-q>d.$$

The proof follows the same lines as in \cite{MR1408433,vage1},
and originates with the work of V{\aa}ge \cite{vage96}. First, notice that
\[
\un d_{n+m, \un \alpha +\un \beta}=b_{n+m}\un a_{\un\alpha+\un\beta} \geq b_{n} b_{m} \un a_{\un \alpha}\un a_{\un \beta} = \un d_{n, \un \alpha} \un d_{m, \un \beta}.
\]
We have:
\[
\begin{split}
  \| f \ast g \|_p^2 &= \sum_{n=0}^\infty b_n^{-p}\left( \sum_{\substack{\un \alpha +\un \beta = \un \gamma\\ |\un\gamma|=n}}  | f_{ \un \alpha} g_{\un \beta}  |^2 a_{\un \gamma}^{-p}\right)\\
  & \leq \sum_{n=0}^\infty  \left(\sum_{\substack{\un\alpha+\un\beta=\un \gamma \\ |\un\gamma|=n}}\sum_{\substack{\un\alpha^\prime+\un\beta^\prime=\un \gamma \\ |\un\gamma|=n}}
    b^{-p/2}_{|\un\alpha|+|\un\beta|}    |f_{\un\alpha}||f_{\un\alpha^\prime}| \un a^{-p/2}_{\un\alpha} \un a^{-p/2}_{\un\alpha^\prime}
    b^{-p/2}_{|\un\alpha^\prime|+|\un\beta^\prime|}    |g_{\un\beta}||g_{\un\beta^\prime}|\un a^{-p/2}_{\un\beta}\un a^{-p/2}_{\un\beta^\prime}\right)\\
  & \leq \sum_{n=0}^\infty  \left(\sum_{\substack{\un\alpha+\un\beta=\un \gamma \\ |\un\gamma|=n}}\sum_{\substack{\un\alpha^\prime+\un\beta^\prime=\un \gamma \\ |\un\gamma|=n}}
    b^{-p/2}_{|\un\beta|}
    |f_{\un\alpha}||f_{\un\alpha^\prime}| \un a^{-p/2}_{\un\alpha}\un a^{-p/2}_{\un\alpha^\prime}
    b^{-p/2}_{|\un\alpha^\prime|}b^{-p/2}_{|\un\beta^\prime|}
    |g_{\un\beta}||g_{\un\beta^\prime}|\un a^{-p/2}_{\un\beta}\un a^{-p/2}_{\un\beta^\prime}\right)\\
  & = \sum_{\un\alpha,\un\alpha^\prime}   b^{-p/2}_{|\un\alpha|}\un a^{-p/2}_{\un\alpha} b^{-p/2}_{|\un\alpha^\prime|} \un a^{-p/2}_{\un\alpha^\prime}    |f_{\un\alpha}||f_{\un\alpha^\prime}|  \underbrace{ \left(
  \sum_{\substack{\un\gamma\ge\un \alpha\\
      \un\gamma\ge \un\alpha^\prime}} \left(b^{-p/2}_{|\un\gamma|-|\un\alpha|} \right)  \left(b^{-p/2}_{|\un\gamma|-|\un\alpha^\prime|}\right)    |g_{\un\gamma-\un\alpha}||g_{\un\gamma-\un\alpha^\prime}|\un a^{-p/2}_{\un\gamma-
  \un\alpha}\un a^{-p/2}_{\un\gamma-\un\alpha^\prime}\right)}_{\mathbf X}.
\end{split}
\]
Using the Cauchy-Schwarz inequality we have:
\[
\begin{split}
  \mathbf X  &=\sum_{\substack{\un\gamma\ge\un \alpha\\
      \un\gamma\ge \un\alpha^\prime}} \left(b^{-p/2}_{|\un\gamma|-|\un\alpha|} \right)  \left(b^{-p/2}_{|\un\gamma|-|\un\alpha^\prime|}\right)    |g_{\un\gamma-\un\alpha}||g_{\un\gamma-\un\alpha^\prime}|a^{-p/2}_{\un\gamma-
  \un\alpha}a^{-p/2}_{\un\gamma-\un\alpha^\prime}
\\
&\le \left(  \sum_{\substack{\un\gamma\ge\un \alpha\\
      \un\gamma\ge \un\alpha^\prime}} \left(b^{-p}_{|\un\gamma|-|\un\alpha|} \right)|g_{\un\gamma-\un\alpha}|^2\un a^{-p}_{\un\gamma - \un\alpha}\right)^{1/2}
\left(  \sum_{\substack{\un\gamma\ge\un \alpha^\prime\\
      \un\gamma\ge \un\alpha^\prime}} \left(b^{-p}_{|\un\gamma|-|\un\alpha^\prime|} \right)|g_{\un\gamma-\un\alpha^\prime}|^2\un a^{-p}_{\un\gamma - \un\alpha^\prime}\right)^{1/2} \le\|g\|_p^2.
%
%
  %
\end{split}
\]
On the other hand, and using again the Cauchy-Schwarz inequality we have:
\[
\begin{split}
\sum_{\un\alpha,\un\alpha^\prime}   b^{-p/2}_{|\un\alpha|}\un a^{-p/2}_{\un\alpha} b^{-p/2}_{|\un\alpha^\prime|}  \un a^{-p/2}_{\un\alpha^\prime}    |f_{\un\alpha}||f_{\un\alpha^\prime}|&=
\left(\sum_{\un\alpha}   b^{-p/2}_{|\un\alpha|}\un a^{-p/2}_{\un\alpha} b^{-p/2}|f_{\un\alpha}| \right)^2\\
&=  \left(\sum_{\un\alpha}   b^{-q/2}_{|\un\alpha|}b^{(q-p)/2}_{|\un\alpha|}
\un a^{-q/2}_{\un\alpha} \un a^{(q-p)/2}_{\un\alpha} |f_{\un\alpha}| \right)^2\\
&\le\left(\sum_{\un\alpha}\un a_{\un\alpha}^{q-p}b_{|\alpha|}^{q-p}\right)\left(\sum_{\un\alpha}b^{-q}_{|\un\alpha|}\un a_{\un\alpha}^{-q}|f_{\un\alpha}|^2\right)\\
&=A_{p,q}\|f\|_q^2
\end{split}
\]
with
\[
A_{p,q}=\left(\sum_{\un\alpha}\un a_{\un\alpha}^{q-p}b_{|\alpha|}^{q-p}\right)=  \left( \sum_{\un \alpha} \un a_{\un \alpha}^{-(p-q)} \right) \left( \sum_{n} b_{n}^{-(p-q)}\right)<\infty.
\]

In conclusion, we do get an associated strong algebra.
\end{proof}

The freedom in the choice of the sequences $(a_n)$ and $(b_n)$ allows  to adapt the method to various specific situations.


\section{Stochastic processes}
\setcounter{equation}{0}
\label{stochas}
As is well known one can associate to every positive definite function on a set, say $S$, a Gaussian stochastic process indexed by the given set.
This is Lo\`eve's Theorem; see e.g. \cite{neveu68}. To construct the process one can proceed as follows: Let
$K(t,s)$ be positive definite on the set $S$,
and let $\mathcal H(K)$ be the corresponding reproducing kernel Hilbert space. Let
\[
K(t,s)=\sum_{a\in A}e_a(t)\overline{e_a(s)},
\]
where $(e_a(t))_{a\in A}$ (where $A$ need not be countable) is an orthonormal basis of $\mathcal H(K)$.
Build a probability space $\Omega=\prod_{a\in A}\Omega_a$ where we can  define a family of independent $N(0,1)$ variables $(Z_a)_{a\in A}$. See
\cite[pp. 38-39]{neveu68} for the latter.
The second order stochastic process
\[
X_t(\w)=\sum_{a\in A} e_a(t)Z_a(\w)
\]
has covariance function $K(t,s)$.\\

The connection with the present (non even Gaussian setting) is to assume $A$ countable, take $\Omega$ to be the probability space
$(\mathcal S^\prime(\R),\mathcal C, dP_\varphi)$ and to replace the orthonormal random variables $Z_a$ by
\[
Q_j=Q^{(0,0,\ldots, 0,1,0,0,\ldots)}
\]
where the $1$ is at the $j$-th coordinate in $(0,0,\ldots, 0,1,0,0,\ldots)$. Note that the $Z_a$ are pairwise independent, since Gaussian, but the $Q_j$
are not independent.\smallskip

We then consider

\begin{equation}
X_t(\w)=\sum_{j=1}^\infty e_j(t)Q_j(\w)
\end{equation}

To make a more precise study one proceeds as follows to construct the process defined by \eqref{qerty123}:

\begin{theorem}
It holds that
\begin{equation}
\sqrt{\varphi_1}X^\varphi_t(\w)=\sum_{j=1}^\infty \left(\int_0^t\zeta_j(v)dv\right)Q_j(\w)
\end{equation}
where the convergence is in $\mathbf L^2(\mathcal S^\prime(\mathbb R),\mathcal C,dP_\varphi)$.
\end{theorem}
\begin{proof}
We have
\[
1_{[0,t]}(u)=\sum_{j=1}^\infty\left(\int_0^t\zeta_j(v)dv\right)\zeta_j(u)
\]
and hence, by the continuity of the isometry $s\mapsto \frac{1}{\sqrt{\varphi_1}}Q_s$
\[
Q_{1_{[0,t]}}(\w)=\sum_{j=1}^\infty  \left(\int_0^t\zeta_j(v)dv\right)Q_j(\w)
\]
where the convergence is in $\mathbf L^2(\cS'(\R),\mathcal C,dP_\varphi)$.
\end{proof}
At this stage we recall the following bounds on the normalized Hermite functions; see \cite[p. 349]{bosw}, \cite[lemma 1.5.1 p.26]{MR1215939}.

\begin{lemma}
There exist strictly positive constants $C$ and $\gamma$, independent of $j$,
and such that
\begin{equation}
  |\zeta_j(t)\le\begin{cases}\, C j^{-1/12}, \hspace{2.2mm}if\,\,\,\, |t|\le 2\sqrt{j},\\
\, Ce^{-\gamma t^2},\,\,\,\,\, if\,\,\,\, |t|> 2\sqrt{j}.
\end{cases}
\end{equation}
\end{lemma}

In particular, there is a strictly positive constant $A$ such that
\[
|\zeta_j(t)|\le A,\quad\forall t\in\mathbb R,\,\, j=1,2,\ldots
\]

We recall that
\begin{equation}
\widehat{\zeta_j}(u)=(-1)^j\zeta_j(u),
\label{place-voltaire}
\end{equation}
where $\widehat{f}$ denotes the Fourier transform; see \eqref{ft}.
This formula shows in particular
that the specific choice of the Hermite functions is crucial. Elements of another orthonormal basis of $\mathbf L^2(\mathbb R,dx)$ made of Schwartz functions need
not satisfy these inequalities. Using \eqref{place-voltaire} and the preceding lemma one proves that (see \cite[(3.11), p. 1089]{aal2} for a more general formula):

\begin{lemma}
The Hermite functions satisfy
\begin{equation}
\label{le-louvre}
|\zeta_j(t)-\zeta_j(s)|\le |t-s|\left(C\sqrt{j}+D\right),\quad j=1,2,\ldots
\end{equation}
for $t,s\in\mathbb R$ and some positive constants $C$ and $D$.
\end{lemma}
\begin{theorem}
Let $b_n=\frac{1}{2^n}$ and $a_n=\frac{1}{2^n}$.
The stochastic process $(X^\varphi_t)_{t\in\mathbb R}$ is differentiable in the space of stochastic distribution $\mathcal H_1(a,b)$.
\end{theorem}

\begin{proof}
Let $t\in\mathbb R$. Our strategy is as follows. We first check that the function
\begin{equation}
\label{is345}
t\mapsto N^\varphi(t)=\frac{1}{\sqrt{\varphi_1}}\sum_{j=1}^\infty \zeta_j(t)Q_j
\end{equation}
is $\mathcal H_1(a,b)$-valued, and show that the difference
\[
\frac{X^\varphi(t+h)-X^\varphi(t)}{h}-N^\varphi(t)
\]
goes to $0$ in the topology of $\mathcal A$. By the properties of the topology in $\mathcal A$, it is enough to consider sequences and to consider
convergence in one of the spaces $\mathcal H_p(a,b)$.\smallskip

More precisely, we have
\[
\|\sum_{j=1}^\infty \zeta_j(t)Q_j\|^2=\sum_{j=1}^\infty |\zeta_j(t)|^2\frac{1}{2^j}\le A\sum_{j=1}^\infty
\frac{1}{2^j}<\infty
\]
and so the function \eqref{is345} is $\mathcal H_1(a,b)$-valued. Moreover, using \eqref{le-louvre} we obtain:
\[
\begin{split}
\left\|\frac{X^\varphi(t+h)-X^\varphi(t)}{h}-N^\varphi(t)\right\|&=\left\|\left(\sum_{j=1}^\infty\frac{\int_t^{t+h}(\zeta_j(s)-\zeta_j(t))ds}{h}\right)Q_j\right\|^2\\
&=\sum_{j=1}^\infty\frac{1}{2^j}\left|\frac{\int_t^{t+h}(\zeta_j(s)-\zeta_j(t))ds}{h}\right|^2\\
&\le\sum_{j=1}^\infty\frac{1}{2^j}\left|\frac{\int_t^{t+h}\left(C\sqrt{j}+D\right)|t-s|ds}{h}\right|^2\\
&\le\sum_{j=1}^\infty\frac{1}{2^j}\left|\frac{\int_t^{t+h}\left(C\sqrt{j}+D\right)|h|ds}{h}\right|^2\\
&=M|h|
\end{split}
\]
with
\[
M=\sum_{j=1}^\infty \frac{(C\sqrt{j}+D)^2}{2^j}<\infty.
\]
\end{proof}

As explained in Remark \ref{r567}, the analysis from \cite{aal2,aal3,ajnfao} can be transferred, {\it mutatis mutandis} to the setting at hand.
We present some of the results, some with outline of proofs and some without proof. We focus on two applications, the first one on giving a model for a (non Gaussian) stationary increments process
and the second on a stochastic integral.

\begin{theorem} (stochastic integral)
Let $f$ be a continuous function from $[0,1]$ into $\mathcal A$. Then the integral
\[
\int_0^1f(t)\lozenge N_\varphi(t)dt
\]
converges in $\mathcal A$.
\end{theorem}

\begin{proof}[Outline of the proof]
The product is jointly continuous (Proposition \ref{rab}) and so the image of the interval $[0,1]$ is compact (since $[0,1]$ is a compact set) and so is inside
one of the spaces $\mathcal H_p(a,b)$, say $\mathcal H_{p_0}(a,b)$. We compute then the integral as a Riemann integral in $\mathcal H_{p_0}(a,b)$,
evaluating the Riemann sums using the V\r{a}ge inequalities. This is equivalent to the convergence in the algebra itself since only sequences are involved.
\end{proof}

\begin{remark}
$f([0,1])$ is compact and so inside one of the $\mathcal H_p(a,b)$, say $\mathcal H_{p_0}(a,b)$ (see Proposition \ref{mur}).
V\r{a}ge inequalities imply that $f(t)\lozenge N^\varphi(t)$ is inside
$\mathcal H_{p_0}(a,b)$ for every $t\in[0,1]$.
\end{remark}

\section{Variation on the main theme, I}
\setcounter{equation}{0}
Let $\mu$ be a positive measure on $\mathbb R$ such that
\[
\int_{\R} \frac{d\sigma(u)}{1+u^2} <\infty, \quad \mbox{for some }p \in \N,
\]
and let $\varphi\in ML$. The function
\begin{equation}
\label{4566546}
\varphi\left(-\frac{\int_{\mathbb R} |s(u)|^2d\mu(u)}{2}\right),\quad s\in\cS(\R),
\end{equation}
is continuous in the Fr\'echet topology. By the Bochner-Minlos theorem there exists a Borel measure $\sf P_{\mu,\varphi}$ such that
\[
\varphi\left(-\frac{\int_{\mathbb R} |s(u)|^2d\mu(u)}{2}\right)=\int_{\cS^\prime(\R)}e^{i\langle\w,s\rangle}d{\sf P}_{\mu,\varphi}(\w),\quad s\in\cS(\R),
\]
We have the isometry
\[
\varphi_1\int_{\mathbb R}|s(u)|^2d\mu(u)=\langle Q_s,Q_s\rangle_{P_{\mu,\varphi}},
\]
where the inner product is in $\mathbf L^2(\mathcal S^\prime(\R),\mathcal C,\sf P_{\mu,\varphi})$.
By density we thus have for $A,B$ Borelian subsets of the real line
\[
\langle Q_{1_A},Q_{1_B}\rangle=\varphi_1\mu(A\cap B).
\]
This formula allows to develop a stochastic integral with respect to the process $(Q_{1_A})$ indexed by the Borel sets of the real line; see \cite{zbMATH06796806} for the Gaussian counterpart.
\section{Variation on the main theme, II}
\setcounter{equation}{0}

Let $\mu$ be a positive measure on $\mathbb R$ satisfying \eqref{4566546}, and let $\varphi\in ML$. The function
\begin{equation}
\label{456654}
\varphi\left(-\frac{\int_{\mathbb R} |\widehat{s}(u)|^2d\mu(u)}{2}\right),\quad s\in\cS(\R),
\end{equation}
is continuous in the Fr\'echet topology; see
\cite[(5.3)-(5.4), p. 517]{MR2793121} for the case $\varphi(z)=e^z$. The proof goes the same way. Thus there exists a probability measure such that
\begin{equation}
\label{123789}
  \varphi\left(-\frac{\int_{\mathbb R} |\widehat{s}(u)|^2d\mu(u)}{2}\right)=
  \int_{\cS^\prime(\R)}e^{i\langle\w,s\rangle}dP_{\mu,\varphi}(\w).
\end{equation}
Furthermore, there exists a continuous operator $X:\mathcal S(\R)\longrightarrow \mathbf L_2(\mathbb R,dx)$ such that
\begin{equation}
\label{mumumu}
\varphi_1\int_{\mathbb R} |\widehat{s}(u)|^2d\mu(u)=\langle X_s,X_s\rangle_{P_{\mu,\varphi}}.
\end{equation}
It follows from \eqref{123789} that the map
\[
\sqrt{\varphi_1}\widehat{s}\mapsto Q_s
\]
is an isometry from $\mathbf L^2(\mathbb R,d\mu)$ into the space $\mathbf L^2(\cS^\prime(\R), \mathcal C,P_{\mu\varphi})$. It extends to the whole of
$\mathbf L^2(\mathbb R,d\mu)$ since the linear span of the functions $u\mapsto \frac{e^{iut}-1}{u}$ is dense in $\mathbf L^2(\mathbb R,dx)$ when $t$ runs
through $\mathbb R$; see e.g. \cite[Exercise 6.3.2 p. 302]{CAPB_2} for a proof of this well known fact.

\begin{theorem}
The process
\[
X_t=Q_{1_{[0,t]}},\quad t\in\mathbb R,
\]
has correlation function
\begin{equation}
\label{totoche}
K(t,s)=\int_{\mathbb R}\frac{e^{itu}-1}{u}\frac{e^{-ius}-1}{u}d\mu(u)
\end{equation}
\end{theorem}

We recall that \eqref{totoche} can be written as
\[
K(t,s)=r(t)+\overline{r(s)}-r(t-s)
\]
with
\[
r(t)=\int_{\mathbb R}\frac{1-e^{itu}}{u^2}d\mu(u).
\]
This includes the fractional Brownian motion (with $d\mu(u)=|u|^{1-2H}, H\in(0,1)$) and the case of singular measures.

\section*{Aknowledgements}
The first author thanks the Foster G. and Mary McGaw Professorship in Mathematical Sciences, which supported this research. The work of the second and third authors was supported by Portuguese funds through the CIDMA -- Center for Research and Development in Mathematics and Applications, and the Portuguese Foundation for Science and Technology (``FCT--Funda\c{c}\~ao para a Ci\^encia e a Tecnologia''), within project references UIDB/04106/2020 and UIDP/04106/2020.  The second author wishes to thank also the support of the Portuguese Foundation for Science and Technology under the FCT  Sabbatical grant ref. SFRH/BSAB/143104/2018.

\bibliographystyle{plain}
\bibliography{all}
\end{document}